\documentclass[11pt]{amsart}
\usepackage{amssymb,hyperref}
\usepackage{graphicx}
\usepackage[matrix,arrow]{xy}
\newcommand{\PP}{\mathbb P}
\newcommand{\kk}{K}

\newcommand{\CC}{{\mathbb C}}

\newcommand{\QQ}{{\mathbb Q}}
\newcommand{\RR}{{\mathbb R}}
\newcommand{\KK}{{K}}

\DeclareMathOperator{\sing}{\operatorname{sing}}

\newcommand{\sss}{{\mathfrak b}}
\newcommand{\ddd}{{\mathfrak c}}
\numberwithin{equation}{section}
\begin{document}

\title{Counting lines on surfaces, especially quintics} % with rational functions}

\author{S\l awomir Rams}
\address{Institute of Mathematics, Jagiellonian University, 
ul. {\L}ojasiewicza 6,  30-348 Krak\'ow, Poland}
\email{slawomir.rams@uj.edu.pl}

\author{Matthias Sch\"utt}
\address{Institut f\"ur Algebraische Geometrie, Leibniz Universit\"at
  Hannover, Welfengarten 1, 30167 Hannover, Germany}

    \address{Riemann Center for Geometry and Physics, Leibniz Universit\"at
  Hannover, Appelstrasse 2, 30167 Hannover, Germany}

\email{schuett@math.uni-hannover.de}

% Even sets of rational curves and intersections of quadrics}
% \author{S{\l}awomir Cynk}
% \address{Institute of Mathematics, Jagiellonian University,
% ul.~Reymonta~4, 30-059 Krak\'ow, POLAND }
% \address{Akademia Pedagogiczna, Instytut Matematyki, ul. Podchor\c a\.zych 2, 30-084 Krak\'ow, POLAND}
% \email{s.cynk@im.uj.edu.pl}
% \author {S{\l}awomir Rams}
% \address{Institute of Mathematics, Jagiellonian University,
% ul.~Reymonta~4, 30-059 Krak\'ow, POLAND}
% \address{Mathematisches Institut der Universit\"{a}t, Bismarckstr. 1 1/2,
% D - 91054 Erlangen, GERMANY}
% \email{rams@mi.uni-erlangen.de}
% \thanks{Suported by the DFG Schwerpunktprogramm
% "Global methods in complex geometry".
% The first named author is partially supported by
% the KBN Grant No.~1~P03A~008~28.
% The second named author is partially supported
% by the  KBN Grant No.~2~P03A~016~25.}
% \subjclass[2000]{Primary: 14J28;  Secondary 14J25}
%% 14J28 = K3 surfaces
%% 14J25 = special surfaces
%% 14N25 = varieties of low degree
% \keywords{even set of rational curves, Mukai correspondence.}

%\author{to be inserted}
%\author {}
% \address{Institut f\"ur Algebraische Geometrie, Leibniz Universit\"at
%  Hannover, Welfengarten 1, 30167 Hannover, Germany \newline
% and \newline
% Institute of Mathematics,
% Jagiellonian University,
% ul. {\L}ojasiewicza 6,
% 30-348 Krak\'ow,
% Poland}

%\email{slawomir.rams@uj.edu.pl}
\date{September 4, 2018}
\thanks{Research partially supported by the National Science Centre, Poland, Opus  grant 
no.\ 2017/25/B/ST1/00853
(S.\ Rams) and by 
 ERC StG~279723 (SURFARI) (M.\ Sch\"utt)}

\begin{abstract}
We introduce certain rational functions on a smooth projective surface $X\subset\PP^3$
which facilitate counting the lines on $X$.
We apply this to smooth quintics in characteristic zero to prove that they contain no more than 127 lines,
and that any given line meets at most 28 others. We construct examples which  demonstrate that the latter bound is sharp. 
\end{abstract}

\maketitle

\newcommand{\XXd}{X_{d}}
\newcommand{\XXf}{X_{4}}
\newcommand{\XXp}{X_{5}}
\newcommand{\mF}{\mathcal F}
\newcommand{\mL}{\mathcal L}
\newcommand{\mR}{\mathcal R}
\newcommand{\Ruledeight}{S_{11}}
\newcommand{\Ruledfour}{S_{4}}
\newcommand{\DivisorRest}{\mathfrak Rest}
\newcommand{\Pl}{\Pi}
\newcommand{\reg}{\operatorname{reg}}
%%% this is sign for the barth function
%\newcommand{\sss}{{\mathfrak b}}
%\newcommand{\ddd}{{\mathfrak c}}
% MS: commented out for now -- seems not necessary (or redundant)

\theoremstyle{remark}
\newtheorem{obs}{Observation}[section]
\newtheorem{rem}[obs]{Remark}
\newtheorem{example}[obs]{Example}
\newtheorem{conv}[obs]{Convention}
\theoremstyle{definition}
\newtheorem{Definition}[obs]{Definition}
\theoremstyle{plain}
\newtheorem{prop}[obs]{Proposition}
\newtheorem{theo}[obs]{Theorem}
\newtheorem{lemm}[obs]{Lemma}
\newtheorem*{conj}{Conjecture}
\newtheorem{claim}[obs]{Claim}
\newtheorem{Fact}[obs]{Fact}
\newtheorem{cor}[obs]{Corollary}
\newtheorem{assumption}[obs]{Assumption}

\newcommand{\ux}{\underline{x}}
\newcommand{\ud}{\underline{d}}
\newcommand{\ue}{\underline{e}}
\newcommand{\mmS}{{\mathcal S}}
\newcommand{\mmP}{{\mathcal P}}
\newcommand{\nlines}{\mbox{\texttt l}(\XXp)}
\newcommand{\ii}{\operatorname{i}}

\newcommand{\nonlinflec}{{\mathcal D}}
\newcommand{\linflec}{{\mathcal L}}
\newcommand{\flec}{{\mathcal F}}
\newcommand{\filename}{wlt-2018-09-04-final.tex}

%%%%%%%%%%%%%%%%%%%%%%%%%%%%%%%%%%%%%%%%%%%%%%%%%%%%%%%%%%%%%%%%%%%%%%%%%%%%%%%%%%%%%%%%%%%%%%%%%%%%%%%%%%%%%%%%%%%%%%%%%%%%%%%%%%%%%%%%%%%%%%%%%%%%%%%%%%%%
%%%%%%%%%%%%%%%%%%%%%%%%%%%%%%%%%%%%%%%%%%%%%%%%%%%%%%%%%%%%%%%%%%%%%%%%%%%%%%%%%%%%%%%%%%%%%%%%%%%%%%%%%%%%%%%%%%%%%%%%%%%%%%%%%%%%%%%%%%%%%%%%%%%%%%%%%%%%
%%%%%%%%%%%%%%%%%%%%%%%%%%%section:  introduction
%%%%%%%%%%%%%%%%%%%%%%%%%%%%%%%%%%%%%%%%%%%%%%%%%%%%%%%%%%%%%%%%%%%%%%%%%%%%%%%%%%%%%%%%%%%%%%%%%%%%%%%%%%%%%%%%%%%%%%%%%%%%%%%%%%%%%%%%%%%%%%%%%%%%%%%%%%%%%
%%%%%%%%%%%%%%%%%%%%%%%%%%%%%%%%%%%%%%%%%%%%%%%%%%%%%%%%%%%%%%%%%%%%%%%%%%%%%%%%%%%%%%%%%%%%%%%%%%%%%%%%%%%%%%%%%%%%%%%%%%%%%%%%%%%%%%%%%%%%%%%%%%%%%%%%%%%%%

\section{Introduction} \label{sect-introduction}

Recently, there has been considerable progress in the 
understanding of configurations of lines on smooth quartic surfaces in $\PP^3$.
If the ground field $\kk$ has $\mbox{char}(\kk) \neq 2,3$, % and  $\kk$ is algebraically closed, 
the maximal number of lines on a smooth $\XXf$ is $64$, and a line can be met by at most $20$ other lines 
(see \cite{Segre}, \cite{RS}). For $\mbox{char}(\kk) = 3$ the maximal number is $112$ 
and a line meets at most $30$ other lines (\cite{rs-2014}).
The cases $\mbox{char}(\kk)=2$ as well as $\kk = \RR$, or even $\kk=\QQ$ 
were studied in \cite{D-2016}, \cite{dis-2015}, \cite{rs-2}.
%Lines on $K3$-quartics with singularities were examined in \cite{veniani-phd}. 
In fact, the former two papers deal more generally, and to great extent,
with smooth quartic surfaces with many lines, not just the maximum. 
%were constructed by various authors and classified in \cite{dis-2015}. 
Most of this progress was made possible by the fact that smooth quartics are K3 surfaces,
 and lines endow them with elliptic fibrations (or genus one fibrations). 
 Moreover, examples can be constructed with the help of  lattice-theoretical  techniques based on the Torelli theorem.   

In contrast, for smooth surfaces $X_d \subset \PP^3$ of degree $d \geq 5$ very little seems to be known,
even if $\kk=\CC$. 
Although the question what is the maximal number of lines appears in various places in the modern literature
(see e.g. \cite{miyaoka-classical}, \cite{miyaoka}, \cite{eisenbud-harris-3264}),
the best known bound essentially stems from work of Salmon (see \cite{salmon2}, \cite{clebsch}) 
and has since then been only slightly improved by Segre (\cite{Segre}):
\begin{eqnarray}
\label{eq:M(d)}
M(d) \leq d(11d-28)+12
\end{eqnarray}
where $M(d)$ is the maximum number of lines on a smooth degree-$d$ surface (at least over $\CC$). 
In comparison, the Fermat surface of degree $d$ contains exactly $3d^2$ lines 
(over any algebraically closed field of characteristic zero or 
exceeding the degree $d$)
and  there are examples with more lines 
known only in degrees $6$, $8$, $12$, $20$ (see \cite{sarti}). %\cite[$\S$~11.2.1]{eisenbud-harris-3264}).  

Even less seems to be known about the maximum \emph{valency} $v(\ell)$ of a line $\ell$ where
\[
v(\ell) = \#\{\text{lines} \; \ell'\subset X_d; \;\; \ell'.\ell=1\}
\]
%i.e. the maximal number $v(\ell)$ of other lines that can meet $\ell$ on $X_d$
and $X_d$ is again assumed to be smooth. 
By \cite{Segre} if $\ell$ is a line of the first kind (to be defined below, see Def.~\ref{def:2nd}),  
$v(\ell)$ cannot exceed $(8d-14)$ (see Proposition \ref{prop-sfk}), but this  bound does not generalize to the so-called lines of the second kind (cf.~\cite{RS}, \cite{ramsdisjoint}).

In the case of quartics, the proof of the sharp bound $v(\ell) \leq 20$ for $\ell \subset \XXf$ in   \cite{RS} 
(outside characteristics $2,3$)
is based on the study of the elliptic fibration given by the pencil of cubics residual
to $\ell$ in $\XXf$. Already for  $\ell \subset \XXp$, i.e.~on quintics,
this approach leads to a genus $3$-fibration, that is much more difficult to control. %here we could quote a paper by catanese pignatelli 
Instead, for a line $\ell \subset \XXd$
we will define certain rational functions $\sss_0$, $\ldots$, $\sss_{d-3}$ on the line $\ell$ % which we shall call Segre functions 
(Def.~\ref{def-bf}) which form the main novelty of this paper. 
The key feature of these functions, especially on the explicit and computational side,
is that their common zeroes encode the points where $\ell$ is met by other lines on $\XXd$
(Prop.~\ref{lemm-fundamental}).

A line of the second kind  on a  quartic surface is always met by at least twelve other lines (see \cite{Segre}, \cite{rs-advgeom}),
whereas one can construct examples of quintic surfaces that contain only one line and the latter is of the second kind. 
In this note we give another  geometric interpretation of the  notion of a line  
of the second kind (see  Prop~\ref{lemm-lines-second-kind}) that explains 
why special properties of lines of the second kind on quartics do not carry over
to lines on  surfaces of degree $d \geq 5$. 
As an application, we sketch an elementary proof of  the valency bound  for lines on quartics
from \cite{RS} in Example~\ref{example-familyZ}.

After those general preparations  we shall focus on the case $d=5$,
i.e.~of quintics $X_5\subset\PP^3$. 
Our first main result concerns the valency on smooth quintics:

\begin{theo}
\label{thm1}
Let $\kk$ be a field of characteristic zero.
If $X_5\subset\PP^3_\kk$ is a smooth quintic containing a line $\ell$,
then
\begin{equation} \label{eq-28} 
v(\ell) \leq 28 \, .
\end{equation}
\end{theo}

%We show (see Thm~\ref{thm-28}) that for a line $\ell \subset \XXp$ the following inequality holds

More precisely, the maximal value can possibly be attained only for lines of  three certain ramification types
(Cor~\ref{cor-28abc}). 
For one of these configurations,
all surfaces  containing it  can be exhibited explicitly using our methods (Example~\ref{example-3squared-28}), 
so the bound from Theorem \ref{thm1} is indeed sharp. 
In particular, this implies that a quintic surface with a fivetuplet  of coplanar lines can contain at most $125$ lines 
(Corollary \ref{cor-125}). 
Combined with a technique using the flecnodal divisor 
%(\cite{salmon2}) 
and some basic topological and intersection-theoretical arguments, this leads to the following bound
for the maximum number of lines on $X_5$ which is our second main result:

\begin{theo}
\label{thm2}
%Let $\kk$ be a field of characteristic $\neq 2,3,5$.
%\marginpar{or $\CC$?}
Let $\kk$ be a field of characteristic zero.
Any smooth quintic $X_5\subset\PP^3_\kk$ may contain at most $127$ lines.
In other words,
$$
M(5)\leq 127.
$$
\end{theo}

%Although this result almost stands in a nice line with those for cubics and quartics,
We have to admit that we do not expect Theorem \ref{thm2} to be sharp, or even close to it.
In fact, the current record for the number of lines on a smooth quintic surface (outside characteristic $2$, $5$) 
stands at $75$, attained both by the Fermat quintic and Barth's quintic (\cite{ramsschuett}).
Yet the given bound provides a substantial improvement compared to \eqref{eq:M(d)}.

% (Prop.~\ref{prop-129lines-general})
%$$
%M(5) \leq 129.
%$$ 

The organization of the paper is as follows.  At first we recall Segre's argument for lines of the first kind
(see Sect.~\ref{sect-ram-type}).
In Sect.~\ref{sect-segre-functions} we define the rational functions $\sss_0$, $\ldots$, $\sss_{d-3}$
 for a line on a smooth degree-$d$ surface 
and  describe their relevant properties. In the subsequent three sections we study lines of the second kind on quintics to 
complete the proof of  
Theorem \ref{thm1}. %\eqref{eq-28}.
Finally, in Sections~\ref{bounding-lines-quintics}, \ref{s:pf} we investigate quintics without fivetuplets of coplanar lines
in order to derive Theorem \ref{thm2}.

% \marginpar{paragraph moved}
%Originally, we introduced the functions $\sss_0$, $\ldots$, $\sss_{d-3}$ to find quartic surfaces with interesting configurations of lines. 
%Indeed, for an explicitly given surface, the $\sss$-functions provide a simple tool to check whether a line is met by many lines.  For quintic surfaces, the
%number of possible ramification types is small enough to allow us to successfully apply a similar approach. Our interst in configurations of lines is justified by their various applications ({\bf here some citations}).

%\noindent
%{\em Convention:} 
\begin{conv}
\label{conv}
Since the statements of this paper remain valid under base extension (or restriction),
we work over an algebraically closed field $\kk$ of characteristic %$\neq 2,3,5$ (?)
zero.
% or greater than $d$ (the degree in question).
%In particular, all results on quintics will hold outside characteristics $2,3,5$.
%(which we suppress in the notation).
\end{conv}
%In this note we work the base field $\CC$.

\subsection*{Acknowledgement}
We thank the anonymous referee for helpful comments.

%%%%%%%%%%%%%%%%%%%%%%%%%%%%%%%%%%%%%%%%%%%%%%%%%%%%%%%%%%%%%%%%%%%%%%%%%%%%%%%%%%%%%%%%%%%%%%%%%%%%%%%%%%%%%%%%%%%%%%%%%%%%%%%%%%%%%%%%%%%%%%%%%%%%%%%%%%%%
%%%%%%%%%%%%%%%%%%%%%%%%%%%%%%%%%%%%%%%%%%%%%%%%%%%%%%%%%%%%%%%%%%%%%%%%%%%%%%%%%%%%%%%%%%%%%%%%%%%%%%%%%%%%%%%%%%%%%%%%%%%%%%%%%%%%%%%%%%%%%%%%%%%%%%%%%%%%
%%%%%%%%%%%%%%%%%%%%%%%%%%%section:  segre functions
%%%%%%%%%%%%%%%%%%%%%%%%%%%%%%%%%%%%%%%%%%%%%%%%%%%%%%%%%%%%%%%%%%%%%%%%%%%%%%%%%%%%%%%%%%%%%%%%%%%%%%%%%%%%%%%%%%%%%%%%%%%%%%%%%%%%%%%%%%%%%%%%%%%%%%%%%%%%%
%%%%%%%%%%%%%%%%%%%%%%%%%%%%%%%%%%%%%%%%%%%%%%%%%%%%%%%%%%%%%%%%%%%%%%%%%%%%%%%%%%%%%%%%%%%%%%%%%%%%%%%%%%%%%%%%%%%%%%%%%%%%%%%%%%%%%%%%%%%%%%%%%%%%%%%%%%%%%

\section{Segre's bound for the lines of the first kind} \label{sect-ram-type}

Let $d \geq 4$ and let $\XXd \subset \PP^3$  be a smooth degree-$d$ surface that contains a line  $\ell$ (i.e. a degree-one curve). 
The linear 
system $|{\mathcal O}_{\XXd}(1) - \ell|$ endows the surface in question with a  
% genus-$\frac{(d-2)(d-3)}{2}$ 
fibration, i.e.~a morphism
\begin{equation} \label{eq:fibration-deg-d}
\pi \, : \,  \XXd \to \PP^1
\end{equation}
whose fibers are plane curves of degree $d-1$.
We %follow  \cite{rs-advgeom} and 
let $F_P$  denote   the fiber of \eqref{eq:fibration-deg-d}
contained in the tangent space $\operatorname{T}_{P} \XXd$  for a point $P \in \ell$.

The restriction of the fibration \eqref{eq:fibration-deg-d} to the line $\ell$ defines 
the degree-$(d-1)$ map 
\begin{equation} \label{eq-map-deg-d}
\pi|_\ell \,: \; \ell \rightarrow \PP^1.
\end{equation}
%which is separable thanks to the assumption on the characteristic (Convention \ref{conv}).
Let $R_{\ell}$ be the ramification divisor of \eqref{eq-map-deg-d}.
By the Hurwitz formula, one has 
$$
\deg(R_{\ell})= 2d - 4 \, .
$$
%because
%$$
%2 \mbox{g}(\ell) - 2 = (d-1) \cdot (2 \mbox{g}(\PP_1) - 2) + \deg(R_{\ell})
%$$
The line $\ell$ is said to be \emph{of ramification type }  $(1^{2d-4})$
if and only if all ramification points are simple,
i.e.~no ramification point occurs in $R_{\ell}$ with multiplicity greater than one.
The lines of other  ramification types are defined in an analogous way;
e.g.  type $(2,1^{2d-6})$ means that only one double point appears in  $R_{\ell}$.
We highlight that
in the case of $d \geq 5$ a phenomenon occurs that was impossible for quartics:
two ramification points of the map \eqref{eq-map-deg-d} may belong to the same fiber 
(see  \ref{example-double-point} for an explicit example).
To simplify our notation we say that the  line $\ell$  is of ramification type 
$$
(1^{2d-6},[1,1])
$$
%\marginpar{exponent corrected, was $[1,1]^2$}
if it is of the type  $(1^{2d-4})$ and exactly two ramification points of index $1$  lie in the same fiber of \eqref{eq-map-deg-d}. 
Moreover, a line  $\ell' \neq \ell$, $\ell' \subset \XXd$ is called \emph{$\ell$-unramified}
if and only if it meets $\ell$ away from the set of ramification  points of the map \eqref{eq-map-deg-d}.

%Thanks to our convention that $\kk$ has characteristic $\neq 2,3$,
The general fiber of \eqref{eq:fibration-deg-d} is a smooth planar curve with $3(d-1)(d-3)$ inflection points.
% where $d=5$.
%hartshorne, p. 305 ex. 2.3.e
In his work on lines on surfaces Segre introduces the following notion 
(cf. \cite[p.~87]{Segre}): 

\begin{Definition}[Lines of Second Kind]
\label{def:2nd}
We call the line  $\ell$ a line of the second kind if and only if 
it is contained in the closure  of the flex locus of the smooth fibers of the fibration
\eqref{eq:fibration-deg-d}. Otherwise,   $\ell$ is called a line of the first kind.
\end{Definition}

Obviously each line $\ell' \neq \ell$ on $\XXd$ that meets $\ell$ is a component of a fiber of \eqref{eq:fibration-deg-d}. In particular, 
it meets $\ell$ in a point where both the equation of the degree-$(d-1)$ curve (= the fiber of the fibration \eqref{eq:fibration-deg-d}) and its hessian vanish.
The resultant of the restrictions of both polynomials to the line $\ell$ is of degree $(8d-14)$ 
in the homogeneous coordinates of $\ell$ (see \cite[p.~88]{Segre}, \cite[Lemma~5.2]{RS}).
After verifying that multiple lines meeting $\ell$ in the same point result in a multiple zero of the resultant,
this yields the following bound for the valency:

\begin{prop}[Segre]  \label{prop-sfk}
If $\ell \subset \XXd$ is a line of the first kind, then it is met by at most $(8d-14)$ other lines on $\XXd$:
\begin{eqnarray}
\label{eq:1st}
v(\ell)\leq 8d-14.
\end{eqnarray}
\end{prop}

%actually examples of \cite{ramsdisjoint} violate the bound of {prop-sfk} for d \geq 9
For quintics, for instance, the bound reads $v\leq 26$ which is a little better than what we stated in Theorem \ref{thm1}.
Indeed, 
examples of smooth degree-$d$ surfaces that contain a line met by $(d(d-2)+2)$ other lines
(thus eventually violating \eqref{eq:1st}) can be found in \cite{ramsdisjoint}, whereas 
quartic surfaces that violate the above bound are studied in \cite{RS}.
It is due to those examples that in this note we will mostly deal with lines of the second kind;
in particular, this will be necessary and sufficient to complete the proof of Theorem \ref{thm1}.

\begin{rem} \label{rem-skcheck-d}
Assume that the ideal of a degree-$d$ surface $X_d\subset\PP^3$ is generated by %the polynomial 
\begin{equation} \label{eq-expansion-d}
f =  \sum_{i,j=0}^{d} \alpha_{i,j} \cdot  x_3^i x_4^j = \sum_{i,j,k,l=0}^{d} \alpha_{i,j,k,l} \cdot  x_3^i x_4^j x_1^k x_2^l \in \kk[x_1, x_2, x_3, x_4] \, , 
\end{equation}
where  $\ell = V(x_3, x_4)$ is a line of the second kind,
$\alpha_{i,j}\in \kk[x_1,x_2]$ is homogeneous of degree $(d-i-j)$, and $\alpha_{i,j,k,l}\in\kk$. 
% Using a coordinate change  we can assume that 
% \begin{equation} \label{eq-alpha01-d}
% \alpha_{1,0}=x_1^4  \mbox{ and } \alpha_{0,1}=x_2^4  \, \, .
% \end{equation}
The degree-$(d-1)$ curve  residual to $\ell$ in the intersection $X_d \cap V(x_4 - \lambda \cdot x_3)$ is given 
by the polynomial 
\begin{equation} \label{eq-flambda-barthfunction}
f_{\lambda}(x_1, x_2, x_3)  := f(x_1, x_2, x_3, \lambda x_3)/x_3
\end{equation}  
We put  $h_{\lambda}(x_1, x_2, x_3)$ to denote  the determinant of the  hessian of $f_{\lambda}$. Let
$r \in \kk[x_1,\lambda]$ be the remainder of division of   $h_{\lambda}(x_1,1,0)$ by $f_{\lambda}(x_1,1,0)$. 
We consider the expansion
\begin{equation} \label{eq-reminder-d}
r = \sum_{i,j} r_{i,j} x_1^i \lambda^j 
\end{equation}
By definition, we have that
\begin{equation} \label{eq-rem2-3-condition}
\ell \mbox{ is of the second kind iff all }  
r_{i,j}  \in \kk[\alpha_{i,j,k,l} ]\mbox{ vanish identically.}  
\end{equation}
%Obviously, we have
%$r_{i,j} \in \KK[\alpha_{i,j,k,l}]$, so one can find all degree-$s$ surfaces with lines of the second kind 
% by solving a system of polynomial equations.
\end{rem}

%%%%%%%%%%%%%%%%%%%%%%%%%%%%%%%%%%%%%%%%%%%%%%%%%%%%%%%%%%%%%%%%%%%%%%%%%%%%%%%%%%%%%%%%%%%%%%%%%%%%%%%%%%%%%%%%%%%%%%%%%%%%%%%%%%%%%%%%%%%%%%%%%%%%%%%
%%%%%%%%%%%%%%%%%%%%%%%%%%%%%%%%%%%%%%%%%%%%%%%%%%%%%%%%%%%%%%%%%%%%%%%%%%%%%%%%%%%%%%%%%%%%%%%%%%%%%%%%%%%%%%%%%%%%%%%%%%%%%%%%%%%%%%%%%%%%%%%%%%%%%%%
%%%%%%%%%%%%%%%%%%%%%%%%%%%%%%%%%%%%%%%%%% Barth functions
%%%%%%%%%%%%%%%%%%%%%%%%%%%%%%%%%%%%%%%%%%%%%%%%%%%%%%%%%%%%%%%%%%%%%%%%%%%%%%%%%%%%%%%%%%%%%%%%%%%%%%%%%%%%%%%%%%%%%%%%%%%%%%%%%%%%%%%%%%%%%%%%%%%%%%%
%%%%%%%%%%%%%%%%%%%%%%%%%%%%%%%%%%%%%%%%%%%%%%%%%%%%%%%%%%%%%%%%%%%%%%%%%%%%%%%%%%%%%%%%%%%%%%%%%%%%%%%%%%%%%%%%%%%%%%%%%%%%%%%%%%%%%%%%%%%%%%%%%%%%%%%

\section{Counting lines with rational functions} \label{sect-segre-functions}

In this section we introduce the main new tool of this note --
the rational functions $\sss_k$ (see Def.~\ref{def-bf}).
The definition is preceded by some elementary lemmata 
which we need in order to show that the functions we introduce are in fact well-defined.

Let $\XXd \subset \PP^3$ %_\KK$ 
be  a smooth  surface of degree $d$ that contains a line $\ell$.
To simplify our notation we assume as in Remark \ref{rem-skcheck-d} that  $\ell = V(x_3, x_4)$  
and let $f$  denote a generator of the ideal $\mbox{I}(X_d)$ with expansion \eqref{eq-expansion-d}. 

We define 
$$
\mathfrak{d}_j(f) := \begin{vmatrix} \frac{\partial^2 f}{\partial x_j \partial x_3} &  \frac{\partial^2 f}{\partial x_j \partial x_4}  \\ 
\frac{\partial f}{\partial x_3} & \frac{\partial f}{\partial x_4} \end{vmatrix}  \, .
$$

Let $P \in \ell$  be a point such that 
$\textstyle{\frac{\partial f}{\partial x_4}}(P) \neq 0$. Since the tangent space $T_P X_d$ can be parametrized by the map 
\begin{eqnarray} \label{parametrizing-tangent-space}
\PP^2 \;\;\;\;\;  & \to &  \;\;\;\;\;\;\; T_PX_d \\
(x_1:x_2:x_3) & \mapsto & (x_1:x_2:x_3:C(P) \cdot x_3)
\;\; \mbox{ where } C(P):= - \textstyle{\frac{\partial f/\partial x_3 (P)}{\partial f/\partial x_4(P)}},\nonumber
\end{eqnarray}
the fiber $F_P$ % of the fibration \eqref{eq:fibration-deg-d} that runs through the point $P$ 
is given by the vanishing of $f_{C(P)}$ (i.e.~\eqref{eq-flambda-barthfunction} with $\lambda = C(P)$). 
One can easily check that for $j=1,2$ % and every point $P' \in \ell$  
the following equality holds
\begin{equation} \label{eq-inerpretdd2}
\textstyle{\frac{\partial f}{\partial x_4}}(P) \cdot \frac{\partial f_{C(P)}}{\partial x_j}|_{\ell} = 
\textstyle{\frac{\partial^2 f}{\partial x_j \partial x_3}}|_{\ell} \cdot \textstyle{\frac{\partial f}{\partial x_4}}(P) -
\textstyle{\frac{\partial^2 f}{\partial x_j \partial x_4}}|_{\ell} \cdot \textstyle{\frac{\partial f}{\partial x_3}}(P).
\end{equation}
Note that the right-hand side gives $\mathfrak d_j(f)(P)$ upon substituting $P$.
As an immediate consequence we obtain the following lemma
(which has an analogous statement for $\mathfrak d_1(f)$ as the proof reflects).

\begin{lemm} \label{lem-d2vanish}
Suppose that $P := (1:p_2:0:0) \in \ell$  is a point such that 
$\frac{\partial f}{\partial x_4}(P) \neq 0$. % $p_2 \neq 0$. 
% $\mathfrak{d}_2(f)(P) \neq 0$. 
Then 
the following conditions are equivalent:
\begin{itemize}
\item  $\mathfrak{d}_2(f)(P) = 0$,
\item $P$ is a ramification point of the map \eqref{eq-map-deg-d}.
\end{itemize}
\end{lemm}
\begin{proof}
The point $P$ is a ramification point of the map \eqref{eq-map-deg-d} if and only if the line $\ell$ is tangent to the fiber $F_P$ in the point $P$, i.e.
\begin{equation} \label{eq-bothpartials}
\textstyle{\frac{\partial f_{C(P)}}{\partial x_1}} (P) = \textstyle{\frac{\partial f_{C(P)}}{\partial x_2}} (P) = 0 \, . 
\end{equation}

Suppose   $\mathfrak{d}_2(f)(P) = 0$.  From \eqref{eq-inerpretdd2}  we obtain $\frac{\partial f_{C(P)}}{\partial x_2}(P) = 0$. 
The Euler identity yields
% $$
% \frac{\partial f_{C(P)}}{\partial x_1} (P) p_1 + \frac{\partial f_{C(P)}}{\partial x_2} (P) p_2 = 0
% $$
$\frac{\partial f_{C(P)}}{\partial x_1}(P) = 0$.
 
Assume  $P$ is a ramification point of the map \eqref{eq-map-deg-d}. 
From \eqref{eq-inerpretdd2} and \eqref{eq-bothpartials} 
we obtain  $\mathfrak{d}_2(f)(P) = 0$.  
\end{proof}

\begin{lemm} \label{lem-smoothcase}
Let $X_d \subset \PP^3$ be a smooth surface that contains the line $\ell = V(x_3, x_4)$ and let $\operatorname{I}(X_d) =  \langle f \rangle$.
Then $f$ satisfies the following conditions
\begin{equation} \label{eq-nonvanishingpartial}
\textstyle{\frac{\partial f}{\partial x_4}} \notin \operatorname{I}(\ell) \quad \mbox{ and }  \quad \mathfrak{d}_2(f) \notin  \operatorname{I}(\ell) \,  .
\end{equation}
% \begin{equation} \label{eq-nonvanishingd2}
%  \mathfrak{d}_2(f) \notin  \mbox{I}(\ell) 
% \end{equation}
\end{lemm}
\begin{proof} Suppose to the contrary that the formal partial derivative $\frac{\partial f}{\partial x_4}$ 
belongs to the ideal $\mbox{I}(\ell)$. 
Since $f \in \mbox{I}(\ell)$, we can write
$$
f = x_3 \cdot g(x_1, x_2, x_3, x_4) + x_4 \cdot h(x_1, x_2, x_4)
$$
Thus, we have $\frac{\partial f}{\partial x_4} \in \mbox{I}(\ell)$ if and only if
 $h(x_1, x_2, x_4) \in \mbox{I}(\ell)$. The latter implies that 
the partials $\frac{\partial f}{\partial x_1}$,  $\frac{\partial f}{\partial x_2}$,  $\frac{\partial f}{\partial x_4}$ vanish along $\ell$.
Hence, the zeroes of $g$ on $\ell$ are singularities of $X_d$. That contradiction shows that $\frac{\partial f}{\partial x_4}$ cannot belong  to the ideal of the line $\ell$. 
% (alternatively - one can see that the Gauss map of $X_d$ contracts $\ell$.) 

Suppose that $\mathfrak{d}_2(f) \in  \mbox{I}(\ell)$.                  % vanishes along the line $\ell$.
 By Lemma~\ref{lem-d2vanish} every point of the line $\ell$
where $\frac{\partial f}{\partial x_4}$ does not vanish  is a ramification point of the map \eqref{eq-map-deg-d}. Contradiction.
\end{proof}

After those preparations we can define   the main tool of this note (recall that the monomial  $x_3^i x_4^j x_1^k x_2^l$ appears  in  \eqref{eq-expansion-d} with the coefficient  $\alpha_{i,j,k,l}$).

\begin{Definition} \label{def-bf}
Let $\ell = V(x_3,x_4)$ be  a line on a smooth surface $\XXd \subset \PP^3$ 
and let   $f$ be a generator of the ideal $\operatorname{I}(X_d)$ given as \eqref{eq-expansion-d}.
For $k=0, \ldots, d-3$
we define the 
$k$-th function $\sss_k := \sss_k(f) \in \kk(x_1, x_2)$ % of  the pair $(f, \ell)$  
as the rational function given 
by the formula
\begin{equation*} \label{eq-segre-function}
\sum_{\substack{j_1, j_2, j_3 \geq 0 \\  j_1 + j_2 \leq k+3\\j_3 \leq d - (k+3)}} 
 {d - (j_1+j_2+j_3) \choose k+3 - (j_1 + j_2)} \alpha_{j_1,j_2,j_3,d-(j_1+j_2+j_3)} A^{d-(k+3) - j_3} B^{(k+3)-(j_1+j_2)} C^{j_2}  
\end{equation*}
where 
\begin{equation*}  \label{eq-ABC}
 A := - \textstyle{\frac{\mathfrak{d}_1(f)}{\mathfrak{d}_2(f)}}|_{\ell} ,  \quad  B := \frac{(\partial f/\partial x_3) \cdot \mathfrak{d}_4(f) - (\partial f/\partial x_4) \cdot  \mathfrak{d}_3(f)}{2 \, \mathfrak{d}_2(f) \cdot (\partial f/\partial x_4)}|_{\ell} ,  \quad 
C:= - \textstyle{\frac{\partial f/\partial x_3}{\partial f/\partial x_4}}|_{\ell} \, .
\end{equation*}
\end{Definition}

Originally, we introduced the functions $\sss_0$, $\ldots$, $\sss_{d-3}$ to find quartic surfaces with interesting configurations of lines. 
Indeed, for an explicitly given surface, the $\sss$-functions will provide a simple tool to check whether a line is met by many lines
(see Proposition \ref{lemm-fundamental} and Example \ref{example-familyZ}).  
For quintic surfaces, the
number of possible ramification types will still be small enough to successfully apply a similar approach. 
This will be crucial for the proof of Theorem \ref{thm1}.
%Our interst in configurations of lines is justified by their various applications ({\bf here some citations}).

The main feature of the rational functions $\sss_k$ will be presented in Proposition \ref{lemm-fundamental}.
Before we get there, we discuss some basic properties.

\begin{rem} \label{rem-coord}
Given a pair $(\XXd, \ell)$ where $\ell \subset \XXd$ is a line, the functions $\sss_k$ do depend on the choice 
of four ordered linear forms $h_1, h_2, h_3, h_4$ 
that constitute a system of homogenous coordinates on $\PP^3$ such that 
$\langle h_3, h_4 \rangle = \operatorname{I}(\ell) $. 
Indeed, one can easily see that  some poles of $\sss_k(f)$ move when we change the homogenous coordinates. Still, as we are about to show,  both the set of common zeroes of 
$\sss_0(f)$, $\ldots$, $\sss_{d-3}(f)$ and the vanishing of $\sss_0(f)$, $\ldots$, $\sss_{j}(f)$ along the line $\ell$ 
have purely geometric meanings
for the pair $(\XXd, \ell)$. We prefer to define $\sss_k$ for fixed homogenous coordinates  $(x_1, x_2, x_3, x_4)$ and a polynomial $f \in \langle x_3, x_4 \rangle$ 
to keep our exposition concise. 
\end{rem}

\begin{rem}
In the  definition of the functions  $\sss_k$ we assumed the surface $\XXd$ to be {\sl smooth}. 
Obviously,   one can define the functions $\sss_k$ 
as soon as
there exists a system of homogenous coordinates $(x_1, x_2, x_3, x_4)$ on $\PP^3$ 
such that the pair $(\XXd,\ell)$ satisfies the conditions  \eqref{eq-nonvanishingpartial}.
This can be used in study of configurations of lines on singular surfaces (cf.~\cite{veniani-phd}).
\end{rem}

For later use, we continue with 
another interpretation of the polynomials $\mathfrak{d}_1(f), \hdots, \mathfrak{d}_4(f)$.
For a point $P \in X_d:=V(f)$, we follow \cite{rs-2014} to define the 
\emph{Hessian quadric} $V_P := V_P X_d$  % %compare 112 lines p. 3
as the quadric in $\PP^3$ given by the quadratic form  %q_P := 
\begin{eqnarray}
\label{eq:q_P}
q_P = \frac 12
(x_1, x_2, x_3, x_4) \Big(   \frac{\partial^2 f}{\partial x_i \partial x_j}(P) \Big)_{1 \leq i,j \leq 4} \phantom{l}^t(x_1, x_2, x_3, x_4).
\end{eqnarray}
Suppose that $P \in \ell$ is  a point such that $\textstyle{\frac{\partial f}{\partial x_4}(P)} \neq 0$. 
As one can easily check,
 the pre-image of  the intersection of the quadric $V_P$  with the tangent space $T_P X_d$ under the parametrization 
\eqref{parametrizing-tangent-space}
is the quadric given  by the $(3 \times 3)$-matrix
\begin{equation} \label{eq-resquadric}
\begin{pmatrix} 
0 & 0  & \mathfrak{d}_1(f)(P)   \\ 
0  & 0  & \mathfrak{d}_2(f)(P)     \\ 
\mathfrak{d}_1(f)(P)  & \mathfrak{d}_2(f)(P)  &  (-B(P)) \mathfrak{d}_2(f)(P)     \\
  \end{pmatrix} 
\end{equation}
In particular, we obtain the following useful observation:

\begin{obs} \label{obs-deg-hessian}
If $P \in \ell$ is a point such  that 
$
\textstyle{\frac{\partial f}{\partial x_4}}(P) \neq 0 \mbox{ and } \mathfrak{d}_2(f)(P) \neq 0
$ 
then the  Hessian quadric $V_P \XXd$ does not contain the tangent space $T_P X_d$.
\end{obs} 

%Indeed, if $P\in\ell\subset V_P$, then $\ell.X_d\geq 3$,
%and generally there are two such lines already contained in $X_d$
%(while lines with $\ell.X_d\geq 4$ yield points on the flecnodal locus).
%\marginpar{added}
%

Finally we can state the proposition  that justifies our interest in the rational functions $\sss_0(f)$, $\ldots$, $\sss_{d-3}(f)$   (and the way we defined them).

\begin{prop} \label{lemm-fundamental} 
Let $X_d \subset \PP^3$ be a smooth surface that contains the line $\ell = V(x_3, x_4)$ and let $I(X_d) =  \langle f \rangle$.
Suppose that $P := (1:p_2:0:0)$  is a point such that 
$$
\textstyle{\frac{\partial f}{\partial x_4}}(P) \neq 0 \;\;  \mbox{ and } \;\;  \mathfrak{d}_2(f)(P) \neq 0.
$$ 
Then there is at most one line  $\ell' \neq \ell$ such that $P \in \ell' \subset X_d$,
and the following conditions are equivalent:
\begin{enumerate}
\item[(a)]  there exists a line $\ell' \neq \ell$ such that $P \in \ell' \subset \XXd$,
\item[(b)] all  functions $\sss_j$ vanish at $P$, i.e. 
$$
\sss_0(P) = \ldots = \sss_{d-3}(P) = 0 \, .
$$
\end{enumerate}
\end{prop}

\begin{proof} 
We assumed the surface $X_d$ to be smooth, so if we have two lines  $\ell' \neq \ell''$ on $\XXd$ such
 that  $\ell',\ell'' \neq \ell$ and   $P \in \ell', \ell''$, 
then the lines in question are coplanar, so  $P$ is a ramification point of the map \eqref{eq-map-deg-d}. The latter is impossible by Lemma~\ref{lem-d2vanish}.

By Observation~\ref{obs-deg-hessian} the Hessian quadric  $V_P$ does not contain the tangent space $T_P X_d$, 
so the latter and $V_P$  meet along two lines, one of which is the line $\ell$.    
An elementary computation (see \eqref{parametrizing-tangent-space} and \eqref{eq-resquadric}) 
shows that  the other line 
%residual to $\ell$ in the intersection of $V_P$ 
%with the tangent space $T_P X_d$ 
can be parametrized by 
the map $\Phi := \Phi_P$ given as %that maps the point $(t_1: t_3) \in \PP^1$ to
\begin{equation*} \label{eq-principal-line}
\PP^1 \ni (t_1: t_3) \mapsto ( t_1 : A(P) \cdot t_1 +  B(P) \cdot t_3 : t_3 : C(P) \cdot t_3) 
\end{equation*}
%\marginpar{how is this defined? $A, B, C$, are not constants, are they? so how does this evaluate really? (and why does $(1:0)$ map to $P$?}
Observe that we have in particular $\Phi(1:0) = P$ (e.g. by \eqref{eq-resquadric}). % is the image of $(1:0)$ under the parametrization $\Phi$.

The line $\Phi(\PP^1)$  meets $X_d$ with multiplicity at least $3$ in the point $P$, so the composition $(f \circ \Phi)$ has  a triple root in the point $(1:0)$. By direct check, one has
\begin{equation} \label{eq-fund-equality}
f \circ \Phi = t_3^3 \cdot \sum_{j=0}^{d-2} \, \sss_j(P) \cdot t_3^{j} \cdot t_1^{d-3-j}  
\end{equation}
so the line $\Phi(\PP^1)$  lies on $X_d$ if and only if {(b)}  holds. 
\end{proof}

In particular, since the surface $X_d$ contains only finitely many lines for any $d\geq 3$, 
Proposition \ref{lemm-fundamental} shows that the $\sss_k$-functions cannot all vanish identically:

\begin{obs} 
If $X_d$ is smooth and $d \geq 3$, then 
there is an integer $k \in \{0, \ldots, d-3\}$
such that
$\sss_k \not\equiv 0$.
\end{obs}

The proof of Prop.~\ref{lemm-fundamental} enables us to reinterpret Segre's notion of the line of the second kind (Def. \ref{def:2nd}).  

\begin{prop} \label{lemm-lines-second-kind}
Let $\XXd \subset \PP^3$ be a smooth surface that contains a line $\ell$.  %= V(x_3, x_4)$ and let $I(X_d) =  \langle f \rangle$.
The following conditions are equivalent:

\begin{enumerate}
\item[(a)]
 $\ell$ is  a line of the second kind,

\item[(b)]
 the function $\sss_0(f)$ vanishes along the line $\ell$. % (i.e. $\sss_0 = 0 \in \CC(x_1, x_2)$). 
\end{enumerate}
\end{prop}

\begin{rem}
Observe that the claim of Prop.~\ref{lemm-lines-second-kind} does not depend on the choice of homogenous coordinates on $\PP^3$
(cf.  Remark~\ref{rem-coord}).
\end{rem}

\begin{proof} We maintain the notation of the proof of Prop.~\ref{lemm-fundamental}. 

Let us assume that the map \eqref{eq-map-deg-d} is not ramified at $P$.
 By definition, we have  $\Phi_P(\PP^1) \neq \ell$ and the line $\Phi_P(\PP^1)$
is tangent to $F_p$ in the point $P$. 
The point $P$ is an inflection point of the fiber in question if and only if
 $(f \circ \Phi_P)$ has a zero of multiplicity $\geq 4$ in the point $(1:0)$. By \eqref{eq-fund-equality} the latter means that 
  $\sss_0(P) = 0$. This completes the proof of the proposition. 
% for a line l' not in the tangent space TpXd the contact is 1 so such a line cannot be an inflection line of the fiber F_p in the point P
% for a line l' \subset TpXd, l' neq \ell  we have i(F_p, l';P)= i(\ell+ Fp, l';P) - 1= mu(f|_{l'},P) - 1 
%  if l' is not a principal line we have mu(f|_{l'},P) = 2 so nothing interesting happens  
% thus P is an inflection point of F_p if and only if the line l' has triple contact with F_p
\end{proof}

Obviously, for a fixed line $\ell \subset X_d$ we can choose the homogenous coordinates $(x_1, x_2, x_3, x_4)$ in such a way 
that the hyperplane section $V(x_3) \cap \XXd$, say,
contains  no lines and no lines on $\XXd$ run through 
the intersection point $V(x_1) \cap \ell$ (which is missing from the description in Prop.~\ref{lemm-fundamental}). 
Once we make such a choice, by Prop.~\ref{lemm-fundamental}, the number of 
$\ell$-unramified lines on $\XXd$ cannot exceed the minimum of the degrees of the
numerators of the non-zero functions $\sss_k(f)$, where $k =0, \ldots, d-3$. Thus the bound on the number of 
$\ell$-unramified lines on $\XXd$ depends on the presentations
\begin{equation} \label{eq-fraction-decomp}
\sss_j(1,x_2) =  \frac{\tilde{\sss}_j}{\ddd_j}   \, .
\end{equation}
where  $\tilde{\sss}_j$, $\ddd_j$ $\in \kk[x_2]$ are relatively prime.

As an illustration, we compute  the functions $\sss_0$, $\sss_1$ for the quartic surfaces from the family ${\mathcal Z}$ 
which was central for \cite[$\S~4$]{RS}.

\begin{example} \label{example-familyZ} 
Let $\XXf$ be a quartic that contains a line $\ell$ of the second kind of ramification type $(2^2)$.
For now, let $\KK$ denote an algebraically closed field of characteristic $\neq 2,3$.
By an elementary computation (see \cite[Lemma~4.5]{RS}) we can assume that $\XXf$ is  given by the equation
\begin{equation*} \label{eq-family-Z}
 x_3x_1^3+x_4x_2^3+x_1x_2q(x_3,x_4)+p(x_3,x_4), 
\end{equation*}
where $q=\sum_{j=0}^2 q_j x_3^j x_4^{2-j}$, $p=\sum_{j=0}^4 p_j x_3^j x_4^{4-j}$ and   $\ell := V(x_3, x_4)$. 
The ramification points are  $P_1 :=(1:0:0:0)$, $P_2 := (0:1:0:0)$ and we have  $T_{P_j}\XXf = V(x_{j+2})$. 
We obtain $\sss_0 = 0$ and
\begin{eqnarray*}  \label{eq-barth-poly-quartic}
\sss_1 & = & \textstyle{\frac{1}{27 x_1^3 x_2^{15}}} \cdot  
( 
-q_2^{3}x_2^{18}+ \left( 27\,p_4+3\,q_1\,q_2^{2} \right) x_2^{15} x_1^3   \\
&  &  - \left( 27\,p_3+3\,q_0\,q_2^{2}+3\,q_1^{2}q_2 \right) x_2^{12} x_1^6 +
\left( q_1^{3}+6\, q_0\,q_1\,q_2+27\,p_2 \right) x_2^{9} x_1^9  \\
& & - \left( 3\,q_0\,q_1^{2}+27\,p_1+3\,q_0^{2}q_2 \right) x_2^{6} x_1^{12} + 
\left( 3\,q_0^{2}q_1+27\,p_0 \right) x_2^{3} x_1^{15} - q_0^{3} x_1^{18})
\end{eqnarray*}
By Prop.~\ref{lemm-fundamental} the line $\ell$ is met by exactly $18$ $\ell$-unramified lines provided  $q_2 q_0 \neq 0$.

We want to check what happens when the ramified fibers of \eqref{eq:fibration-deg-d} contain lines.
 One can easily see that 
the fiber  $F_{P_1}$ contains exactly
one (resp. three) lines if and only if $p_0=0$ (resp. $q_0=0$). 
 A similar argument shows that $F_{P_2}$ contains exactly one (resp. three) lines if and only if
$p_4=0$ (resp. $q_2=0$). 

%The formula for $\\sss_1$ implies that 
Whenever a ramified fiber (resp. both ramified fibers) acquires a triplet of lines, Prop.~\ref{lemm-fundamental} implies that 
the number of $\ell$-unramified lines drops by three (resp. six) and we arrive at exactly $18$ lines that meet $\ell$ if $p_0p_4\neq 0$,
or at least $19$ if  $p_0p_4= 0$. 
%\marginpar{case 19 added}
The maximal number $20$
is attained precisely when both $p_0$, $p_4$ vanish, but no $q_0, q_2$. 
In this case, the line $\ell$ is met by $20$ other lines on the quartic $\XXf$: exactly one line through 
each of ramification points and $18$ $\ell$-unramified lines.  
\end{example}

In Example~\ref{example-familyZ}  we gave an elementary proof of  the most important part of \cite[Prop.~1.1]{RS}.
The computations we just carried out depict a useful property. Apparently the presentation \eqref{eq-fraction-decomp}  
can be used only  to count  the $\ell$-unramified lines. Still, 
the functions $\sss_0$, $\ldots$, $\sss_{d-3}$
detect the existence of many lines through the ramified points of \eqref{eq-map-deg-d} as well
(i.e. the degree of the numerator of $\tilde{\sss}_j$  decreases when many lines on $X_d$ 
run through a ramification point of the map \eqref{eq-map-deg-d}).
On the other hand, the functions $\sss_0$, $\ldots$, $\sss_{d-3}$ can also fail 
to detect the existence of a line $\neq \ell$ through a ramification point  
(see ~\ref{subsec-twothree}).

%%%%%%%%%%%%%%%%%%%%%%%%%%%%%%%%%%%%%%%%%%%%%%%%%%%%%%%%%%%%%%%%%%%%%%%%%%%%%%%%%%%%%%%%%%%%%%%%%%%%%%%%%%%%%%%%%%%%%
%%%%%%%%%%%%%%%%%%%%%%%%%%%%%%%%%%%%%%%%%%%%%%%%%%%%%%%%%%%%%%%%%%%%%%%%%%%%%%%%%%%%%%%%%%%%%%%%%%%%%%%%%%%%%%%%%%%%%
%%%%%%%%%%%%%%%%%%%%%%%%%%%%%%%%%% second kind on quintics
%%%%%%%%%%%%%%%%%%%%%%%%%%%%%%%%%%%%%%%%%%%%%%%%%%%%%%%%%%%%%%%%%%%%%%%%%%%%%%%%%%%%%%%%%%%%%%%%%%%%%%%%%%%%%%%%
%%%%%%%%%%%%%%%%%%%%%%%%%%%%%%%%%%%%%%%%%%%%%%%%%%%%%%%%%%%%%%%%%%%%%%%%%%%%%%%%%%%%%%%%%%%%%%%%%%%%%%%%%%%%%%%%%%%%%%

\section{Lines of the second kind on quintic surfaces} \label{sect-sec-kind-quintics}

Let $\XXp \subset \PP^{3}$ be a { smooth} quintic surface that contains a line $\ell$. 
In this case we obtain the degree-$4$ map \eqref{eq-map-deg-d} and the ramification divisor $R_{\ell}$ is  of 
degree six,  so there are seven  possible ramification types of a line
(plus the subcases when two ramification points lie in the same fiber). 
In this section we discuss important properties of quintic surfaces with lines of the second kind. 

If $P\in\ell$ is an $n$-fold ramification point, 
then generally there can be at most $n+1$ other lines on $\XXp$ containing $P$.
For simple ramification points on a line of the second kind, 
there is a useful strengthening
(which is a  consequence of \eqref{eq-rem2-3-condition}).
%Remark~\ref{rem-skcheck-d}).

\begin{lemm} \label{lem-1line}
Let $\ell$ be a line of the second kind on a smooth quintic $\XXp$ and let $P$ occur in the divisor  $R_{\ell}$ with multiplicity one.
Then at most two lines on $\XXp$ run through the point $P$ (one of them, by assumption, being $\ell$). 
\end{lemm}
\begin{proof} We maintain the notation of Remark~\ref{rem-skcheck-d}, put $P= (1:0:0:0)$, 
\begin{equation} \label{eq-lem41}
\alpha_{1,0}=x_1^2 (x_1- c_1 x_2)(x_1 - c_2 x_2), \quad \alpha_{0,1}=x_2^2(x_2- d_1 x_1)(x_2- d_2 x_1) \, . 
\end{equation}
and assume, after some linear transformation, that the line $\ell':=V(x_2, x_3)$ lies on the quintic surface $\XXp$. % (i.e. some coeffcients vanish in \eqref{eq-expansion-d}). 
By assumption, we have  $d_1 d_2 \neq 0$, since otherwise $P$ would not be a simple ramification point.
Hence  % Remark~\ref{rem-skcheck-d}
\eqref{eq-rem2-3-condition} yields $\alpha_{0221}=\alpha_{0311}=0$.  One can easily check that the line $\ell'$
is the (reduced) tangent cone $C_{P}F_{P}$.
\end{proof}

%All quartic surfaces with lines of the second kind were studied in 
% Remark~\ref{rem-skcheck-d} 
The condition \eqref{eq-rem2-3-condition} 
reduces the problem of finding surfaces with lines of the second kind to
solving (relatively large) systems of polynomial equations.   In the remainder of this section we explain how one can find
all quintics with a line $\ell$ of the second kind such that
\begin{equation} \label{eq-rrl-five}
R_{\ell} \in \{ (2^2,1^2), \, (2^3), \, (3,1^3), \, (3,2,1), \, (3^2)  \} \, .
\end{equation}
The remaining two ramification types, namely $(1^6)$ and $(2,1^4)$, are harder to parametrize,
so we will use different methods to deal with them (in the next section) in order to prove Theorem \ref{thm1}.

Let us assume that $\XXp$ is given by \eqref{eq-expansion-d}  with $d=5$, $\ell = V(x_3, x_4)$ and $r$ is the remainder 
defined in Remark~\ref{rem-skcheck-d}.  
One can easily check that each % coefficient
 $r_{i,j}$ 
\begin{itemize}
\item is of degree at most $2$ with respect to  $\alpha_{i,j,k,l}$, when $i+j=2$,
\item is  linear  with respect to   $\alpha_{i,j,k,l}$, for $i+j=3$,
\item is constant with respect to $\alpha_{i,j,k,l}$ with $i+j \geq 4$. 
\end{itemize}

By a linear  coordinate change that does not alter the ramification type of $P$ 
(of the form  $\tilde{x}_i= x_i+ \sum_{j=3,4} b_{i,j} x_j$ where $i=1,2$, $j=3,4$),
we can assume that four of the coefficients $\alpha_{i,2-i,j,k}$ vanish (the most convenient choice of the four coefficients
will depend on the ramification type we study -- see e.g. \eqref{eq-vanishing-4coeff-33}).  
Moreover, a direct computation shows that   $\alpha_{2003}$, $\alpha_{0230}$ are zero
(a geometric interpretation of this fact for some ramification types can be found in  \cite[$\S$~5]{Segre}). 

Throughout the remainder of  this section, we denote two of the ramification points (after some linear transformation)  by 
$P_1 := P = (1:0:0:0)$, $P_2:= (0:1:0:0)$. 
Then we fix the polynomials  $\alpha_{1,0}$, $\alpha_{0,1}$ appropriate for the given ramification type. 
%(which we will always assume to have another ramification point at $(0:1:0:0)$)
A determinant computation shows that for all quintic surfaces we consider
one can choose twelve $r_{i,j}$ such that the linear system of equations they define 
(with the indeterminates $\alpha_{i,3-i,k,l}$)
has a unique solution. In this way we obtain a map 
\begin{equation} \label{eq-paraRl}
\kk^6 \times \kk^{16} \rightarrow  {\mathcal O}_{\PP^3}(5)
\end{equation}
%I have 12 $\alpha_{i,j,k,l}$, when $i+j=2$, $k+l=3$ but 4+2 vanish
%I have  10 $\alpha_{i,j,k,l}$, when $i+j=4$, $k+l=1$
%I have  6 $\alpha_{i,j,k,l}$, when $i+j=5$, $k=l=0$          
that associates to $\alpha_{i,j,k,l}$, where $(i+j) \in \{2,4,5\}$,  the quintic given by \eqref{eq-expansion-d}. 

The substitution of \eqref{eq-paraRl} into all $r_{i,j}$'s yields several affine quadrics in the affine parameter space $\kk^6$ 
which corresponds to $\alpha_{i,2-i,k,l}$.
Let us denote their (set-theoretic)  intersection by ${\mathcal W}$. 
By definition, the quintics with a line of the second kind of the considered ramification type
can be parametrized by the image of the restriction of \eqref{eq-paraRl} to 
$$
{\mathcal W} \times \kk^{16} 
$$
It should be pointed out that in all cases we study ${\mathcal W} \subset \kk^6$ happens to be an affine subspace.

Below we discuss the ramification types from \eqref{eq-rrl-five} one-by-one in a series of subsections. 
We omit certain formulae to keep our exposition compact, 
but all formulae can be  obtained with the help of any computer algebra system without difficulty.
%(for instance Maple~2016). 

%\vspace*{1ex}
%\noindent
%{\sl Notation.} %In all examples in this section $\XXp$ is a smooth quintic surface given by  \eqref{eq-expansion-d}. Moreover 

\begin{rem}
A careful analysis, for instance based on a parameter count,
reveals that, just like for quartics \cite{rs-advgeom},
quintics with a line of the second kind do not form a single irreducible family,
as different ramification types can yield distinct families.
This can also be inferred from the degree of a certain  surface in $\PP^3$ ($S^r_{11}$ to be investigated in Section \ref{s:princ},
see Remark \ref{rem:deg9}).
\end{rem}

\subsection{Two triple ramification points} \label{example-one-triple-point}
Let $\alpha_{1,0}=x_1^4$ and $\alpha_{0,1}=x_2^4$.  In particular we have $R_{\ell} = 3(P_1 + P_2)$  and $T_{P_i}\XXp = V(x_{i+2})$   
for both ramification points. \\
After a linear change of coordinates we can assume  that 
\begin{equation} \label{eq-vanishing-4coeff-33}
\alpha_{1130} = \alpha_{1103} = \alpha_{2030} = \alpha_{0203} = 0.
\end{equation} 
We consider the system of equations given by  $r_{0,2}$, $\ldots$, $r_{0,5}$ and $r_{i,j}$ where $i=2,3$, $j=1,2,3,4$, to arrive at \eqref{eq-paraRl}.
One can easily check that ${\mathcal W} = V(\alpha_{2012}, \alpha_{0221})$ in this case.

Consequently, every smooth quintic $\XXp$ with a line of the second kind of ramification type $(3^2)$ 
can be given by 
\eqref{eq-expansion-d} where 
\begin{eqnarray} \label{eq-quintic33-sc}
 \alpha_{{3002}}= \textstyle{\frac{1}{8}} {\alpha_{{2021}}}^{2} ,  &  \alpha_{{0320}}= \textstyle{\frac{1}{8}} {\alpha_{{0212}}}^{2},  &  \alpha_{{2102}}= \textstyle{\frac{1}{4}} \alpha_{{1121}}\alpha_{{2021}}, \\
\alpha_{{2120}}= \textstyle{\frac{1}{8}} {\alpha_{{1112}}}^{2} , & \alpha_{{1202}}= \textstyle{\frac{1}{8}} {\alpha_{{1121}}}^{2}, & 
\alpha_{{1220}}= \textstyle{\frac{1}{4}} \alpha_{{0212}}\alpha_{{1112}} \nonumber
\end{eqnarray}
and the other coefficients  $\alpha_{i,3-i,k,l}$ vanish.

\begin{rem} \label{rem-triple-points} 
Suppose that $\ell$ is of ramification type $(3^2)$.
Computing the (formal) Taylor expansion of the equation of $F_{P_1}$ (resp.  $F_{P_2}$)  around $P_1$ (resp. $P_2$) one checks that
the line $\ell$ is the tangent cone of $F_{P_1}$ (resp.  $F_{P_2}$) provided either $\alpha_{0212} \neq 0$ or $\alpha_{0410} \neq 0$ (resp. either $\alpha_{2021} \neq 0$ or 
$\alpha_{4001} \neq 0$). 
\end{rem}

\subsection{One triple ramification point}
\label{example-one-triple-point-b}
In order to study   ramification types with a single triple point 
%$(3,2,1)$, $(3,1^3)$, $(3,[1,1],1)$ 
we put 
$$
\alpha_{1,0}=x_1^4, \quad \alpha_{0,1}=x_2^2(x_2-x_1)(x_2-c x_1)
$$
in \eqref{eq-expansion-d}. As in  \ref{example-one-triple-point},
we can assume that \eqref{eq-vanishing-4coeff-33} holds
and check that the same system of twelve linear equations has a unique solution for every $c \in \kk$.

A discriminant computation reveals that for
three values of $c$ (namely $c=0$  and $9c^2-14c+9=0$), 
we obtain the ramification type $(3,2,1)$, whereas $c=1$ uniquely leads to 
$(3,1,[1,1])$. Otherwise the line $\ell$ is of the (non-degenerate) type $(3,1^3)$ 
(i.e. no pair of ramification points lies in one and the same fiber). 

One can check that for $c=0$ (i.e. ramification type  $(3,2,1)$) the set ${\mathcal W}$ is given by vanishing of $\alpha_{2012}$ and
$$
\alpha_{0212}= - \textstyle{\frac{9}{64}} \alpha_{1121} - \textstyle{\frac{4}{3}} \alpha_{0221} 
- \textstyle{\frac{243}{16384}}  \alpha_{2021} - 
\textstyle{\frac{27}{256}}  \alpha_{1112}
$$
For the values of $c$ such that $\ell$ is not of ramification type $(3,2,1)$, % ramification types,
 ${\mathcal W}$ is a $3$-dimensional affine subspace of $\kk^6$ given, among others, by the vanishing of $\alpha_{2012}$. 
% We skip the long formulae 
% for $\alpha_{2012}$, $\alpha_{0221}
% to keep our exposition compact. 
%\end{example}

\medskip

In the remark below we maintain the notation of \ref{example-one-triple-point}. 
Recall that the fiber $F_{P_1}$ (resp.  $F_{P_2}$) is the quartic curve residual to $\ell$ in the intersection of $\XXp$ with the plane $V(x_3)$ (resp. $V(x_4)$).

\begin{rem} \label{rem-triple-point} 
{\sl a.} Suppose that $\ell$ is of the type $(3,2,1)$. Without loss of generality 
we can assume that $P_1$ (resp. $P_2$) is a double, (resp. a triple) ramification point   and the fiber $F_{P_1}$ meets the line $\ell$ in the point 
$(1:1:0:0)$. This amounts to the choice   $c=0$ in  \ref{example-one-triple-point-b} (in particular, we do not have to consider the other values of $c$
for the ramification type  $(3,2,1)$).  %the other values of c given double ramification point away from $F_{P_1}$  
 Then 
$$R_{\ell} = 2 P_1 + 3 P_2 + P_3, \mbox{ with }  P_3 = (4:3:0:0).$$
As in Remark~\ref{rem-triple-points}, one can check that if the fiber  $F_{P_2}$ contains a line through the point $P_2$, then $\alpha_{2021}$, $\alpha_{4001}$ vanish. 
%By a similar argument the tangent cone $C_{P_3}F_{P_3}$ is a line $\neq \ell$. Thus 
By Lemma~\ref{lem-1line} the quintic $\XXp$ contains at most one line $\neq \ell$ through the point $P_3$.  

\noindent
{\sl b.} For the ramification $(3,1^3)$ (including the case $(3,1,[1,1])$) we have the same condition for  $F_{P_2}$ to have a degree-1 component that contains the triple ramification point $P_2$. 
Lemma~\ref{lem-1line} shows that $\XXp$ contains at most one line $\neq \ell$ through each of the other ramification points.
\end{rem}

\subsection{Double ramification points} 
\label{example-double-point}
Let us assume that 
$$
\alpha_{1,0}=x_1^3(x_1-x_2), \quad \alpha_{0,1}=x_2^3(x_2-c x_1) \;\; \mbox{ and } \; c \neq 0, 1.
$$

A discriminant computation shows that if  $c \neq 0,1,4, -8$, the line $\ell$ is of the non-degenerate 
type $(2^2,1^2)$. Obviously, $c=1$ yields a singular quintic. 
For  $c=4$ 
we have $R_{\ell} = 2 (P_1 + P_2 + P_3)$, where $P_3 = (1:2:0:0)$, so  the line $\ell$ is of ramification type $(2^3)$.           
%here $T_{P_3}\XXp = V(x_3 + 16x_4)$. 
Finally, if $c=-8$,
 then the fiber of the parametrization \eqref{eq-map-deg-d} defined by  the plane $V(x_3 - 64x_4)$ consists 
of exactly  two (different) ramification points. Hence, the line $\ell$ is of  the ramification type $(2^2,[1,1])$.

After some linear transformation,
 we can assume that $\alpha_{2021}$, $\alpha_{2030}$, $\alpha_{1130}$, $\alpha_{0212}$ vanish.
In order to determine the map  \eqref{eq-paraRl}  for a fixed $c \neq 0$, we
solve the system of equations given by the vanishing of $r_{3,0}$, $\ldots$, $r_{3,4}$ and $r_{i,j}$ where $i=1,2$ and $j=1,2,4,5$.

For $c=4$ the intersection of quadrics ${\mathcal W}$ is given by the single equation
\begin{equation} \label{par-222-lastform}
\alpha_{0203} = -128 \alpha_{2012} + 16 \alpha_{1103} + 8 \alpha_{1112} + 4 \alpha_{1121} - \alpha_{0221}/4.
\end{equation}
If $c=-8$ the set  ${\mathcal W}$ is  a codimension-2 vector subspace of $\kk^6$
 cut out by:
\begin{eqnarray*} \label{par-22[1,1]-lastform}
\alpha_{0203} &=& -64 \alpha_{1103} + 16 \alpha_{1112} + 1024 \alpha_{2012}, \\
\alpha_{0221} &=& -128 \alpha_{1112}- 64 \alpha_{1121} - 8192 \alpha_{2012}.
\end{eqnarray*}
For $c \neq 0,1,4, -8$, we obtain a codimension-$2$ affine % here vector space was also ok
 subspace again. Indeed, one can easily check that
$\alpha_{0203}$,  $\alpha_{1103}$ can be uniquely expressed as elements of 
$\kk[c,1/c][\alpha_{0221}, \alpha_{1112}, \alpha_{1121}, \alpha_{2012}]$.
% degree-1 functions of $\alpha_{0221}$, $\alpha_{1112}$, $\alpha_{1121}$, $\alpha_{2012}$
% with coefficients in $\CC[c,1/c]$. % We skip the formulae in this case to keep our exposition compact.  
%%\alpha_{0203} &=& -1/18 1/(c^2(c+8)(c-4) (4(c^2+c+16))(2c^2+5c-16) \alpha_{0221}-81 \alpha_{1112} c^5 -54c^4(c+2) \alpha_{1121}  -27c^6(c-16) a_{2012}
%%\end{eqnarray*}
%\end{example}

\begin{rem} \label{rem-double-point}
%Let us maintain the notation of \ref{example-double-point}.
%
\noindent
{\sl a.} By a tangent cone argument  
if  $a_{0221}\neq 0$ (resp.~$a_{2012} \neq 0$)  then at most one line $\neq \ell$ on $\XXp$ 
runs through the point $P_1$ (resp.~$P_2$). 

\noindent
{\sl b.} A similar argument shows that if  $a_{0221}$, $a_{0410}$, $a_{0401}$ (resp.~$a_{2012}$, $a_{4010}$, $a_{4001}$) vanish simultaneously, then the fiber $F_{P_1}$ (resp.~$F_{P_2}$)
contains no lines at all through the point $P_1$ (resp.~$P_2$). 
\end{rem}

%
%\begin{rem}
%One can easily check that all computations of this section even remain valid over  fields of characteristic $5$. 
%Indeed, all coefficients in the parametrizations  
%we discussed are products of powers of $2$ and $3$.
%\end{rem}

%%%%%%%%%%%%%%%%%%%%%%%%%%%%%%%%%%%%%%%%%%%%%%%%%%%%%%%%%%%%%%%%%%%%%%%%%%%%%%%%%%%%%%%%%%%%%%%%%%%%%%%%%%%%%%%%%%%%%%%
%%%%%%%%%%%%%%%%%%%%%%%%%%%%%%%%%%%%%%%%%%%%%%%%%%%%%%%%%%%%%%%%%%%%%%%%%%%%%%%%%%%%%%%%%%%%%%%%%%%%%%%%%%%%%%%%%%%%
%%%%%%%%%%%%%%%%%%%%%%%%%%%%%%%%%%%%%%%%%% segre's surface of principal lines
%%%%%%%%%%%%%%%%%%%%%%%%%%%%%%%%%%%%%%%%%%%%%%%%%%%%%%%%%%%%%%%%%%%%%%%%%%%%%%%%%%%%%%%%%%%%%%%%%%%%%%%%%%%%%%%%%%%%
%%%%%%%%%%%%%%%%%%%%%%%%%%%%%%%%%%%%%%%%%%%%%%%%%%%%%%%%%%%%%%%%%%%%%%%%%%%%%%%%%%%%%%%%%%%%%%%%%%%%%%%%%%%%%%%%%%%%

\section{Segre's surface of principal lines along a line} \label{sect-segre-surface}
\label{s:princ}

    Having discussed the five ramification types from \eqref{eq-rrl-five},
    it remains to analyse  the ramification types $(1^6)$ and $(2,1^4)$
    (which, as we pointed out before, are relatively hard to parametrize explicitly).
    Their analysis is the content of the present section using an idea that implicitly appears in \cite{Segre}
to count the lines that meet a line of the second kind of the given ramification type. % $(1^6)$ (resp. $(2,1^4)$). 

To this end, we consider the following subset of the Grassmannian Gr$(1,3)$,
\begin{eqnarray*}
{\mathfrak S}_{\ell} &
:= & \{      
\tilde{\ell} \, : \, \tilde{\ell} \mbox{ is a line}, \exists \mbox{ } P \in \ell \mbox{ such that } \tilde{\ell} \subset \operatorname{T}_P \XXp  \mbox{ and }  \ii(\tilde{\ell}, F_{P};P) \geq 2 \, \},
\end{eqnarray*}
where
$F_P$ denotes the residual fiber through $P$ and
$\ii(\cdot)$ stands for the 
intersection multiplicity of the planar curves  $\tilde{\ell}$ and $F_{P}$ in the tangent plane $\operatorname{T}_P\XXp$ 
computed in the point $P$. 
We have the following lemma (maintaining the notation of Remark~\ref{rem-skcheck-d} and \cite[Lemma~2.3]{rs-advgeom}).
% ({\bf this lemma should move to the section with generalities -i.e. th ebound should work for every degree})
\begin{lemm} \label{lemma-ruledoctic}

The total space
%\marginpar{was union}
  $\cup \{ \tilde{\ell} \, : \, \tilde{\ell} \in  {\mathfrak S}_{\ell} \} =: \Ruledeight$
is an algebraic surface of degree at most eleven.
% Moreover, $\ell$ is a directrix of $\Ruledeight$. 
\end{lemm}

\begin{proof} 
%Let us maintain the notation of Remark~\ref{rem-skcheck-d}. In particular, $f$  
%stands for the generator of the ideal of $\XXp$ and we assume that \eqref{eq-expansion-d} holds.
% Observe that we can assume  $\alpha_{1,0}$, $\alpha_{0,1}$  to have no multiple roots. 

Let $P = (p_1: p_2: 0: 0)  \in \ell$ be a point such that $\alpha_{0,1}(p_1,p_2) \neq 0$. 
The residual quartic curve
 $F_{P}\subset T_PX_5$ is given by the polynomial 
$$
g(x_1,x_2,x_3) := f( \alpha_{0,1}(P)x_1, \alpha_{0,1}(P) x_2, \alpha_{0,1}(P) x_3, - \alpha_{1,0}(P) x_3)/x_3 \, .
$$
Indeed, this is just $(\alpha_{0,1}(P)^5 \cdot f_{C(P)})$, where $f_{C(P)}$ is defined by \eqref{eq-flambda-barthfunction}. 

By direct computation, there exist $h_1$, $h_2 \in \kk[x_1, x_2]$ (resp. $h_3$) of degree $7$  (resp. $11$) such that  
\begin{eqnarray*}
\textstyle{\frac{\partial g}{\partial x_3 }}(p_1,p_2,0) &=&  \alpha^3_{0,1}(p_1,p_2) \cdot h_3(p_1,p_2), \\ 
\textstyle{\frac{\partial g}{\partial x_j }}(p_1,p_2,0) &=&  \alpha^4_{0,1}(p_1,p_2) \cdot h_j(p_1,p_2) \;\;
\mbox{ for } j=1,2.
\end{eqnarray*}
We  define bihomogenous polynomials $H_{11} \in \kk[z_1,z_2][x_1, \ldots,x_4]$ (resp. $H_4$)
of bidegree $(11,1)$ (resp. $(4,1)$) 
\begin{eqnarray*}
 H_{11} &:=&  \alpha_{0,1}(z_1,z_2) \cdot h_1(z_1,z_2) \cdot x_1 + 
\alpha_{0,1}(z_1, z_2) \cdot h_2(z_1, z_2) \cdot x_2 + h_3(z_1, z_2) \cdot x_3 \, ,  \\
H_{4} &:=& \alpha_{1,0}(z_1, z_2) \cdot x_3 + \alpha_{0,1}(z_1, z_2) \cdot x_4 \,  ,
\end{eqnarray*}
and repeat almost verbatim the proof of \cite[Lemma~2.3]{rs-advgeom},
dehomogenizing  by putting $z_2 =1$ and computing the resultant with respect to $z_1$.
%. The resultant of $H_{11}$, $H_4$ 
%with respect to $x_1$ gives  $\Ruledeight$
%\marginpar{was part of following Remark, and changed $x_w, x_1$ to $z_2, z_1$}
 (In \cite{rs-advgeom} one deals with a quartic surface $\XXf$ and considers
 bihomogeneous polynomials $H_8$, $H_3$ instead of $H_{11}$, $H_4$, but that's about the only  difference.) 
\end{proof}

\begin{Definition}  \label{def-spl}
The surface $\Ruledeight$ will be called Segre's \emph{surface of principal lines} along $\ell$.
\end{Definition}

\begin{rem} \label{rem-cSspl}
%a) Let  $(\XXp, \ell)$ be as in Remark~\ref{rem-skcheck-d}. The proof of Lemma~\ref{lemma-ruledoctic} (see \cite[Lemma~2.3]{rs-advgeom}
%for more details)
%shows that 
% Segre's surface of principal lines on $\XXp$ along $\ell$
%can be computed in the following way. \\
%We define the hypersurfaces $H_{11}$, $H_4$ as in the proof of Lemma~\ref{lemma-ruledoctic},
%dehomogenize by putting $x_2 =1$. The resultant of $H_{11}$, $H_4$ 
%with respect to $x_1$ gives  $\Ruledeight$.  \\
%b) % For a point $P \in X_d:=V(f)$ we define the hessian quadric $V_P := V_P X_d$  
%compare 112 lines p. 3
%as the quadric given by the quadratic form
%$$
% q_P := (x_1, x_2, x_3, x_4) \Big(   \frac{\partial^2 f}{\partial x_i \partial x_j} \Big)_{1 \leq i,j \leq 4} (x_1, x_2, x_3, x_4)^t
%$$
Suppose $P \in \ell$ and the tangent space $T_P\XXp$ meets the hessian quadric $V_P$ along two lines (obviously one of them  is $\ell$). Then, by definition both lines  
are contained in  $\Ruledeight$.  
\end{rem}

% To study the above two cases we  will apply the following lemma by Segre (c.f.  \cite[$\S$~6.]{Segre}).

Segre's surface of principal lines has the following useful properties 
(the proofs of claims a) and b) use the standard topology of $\PP^3(\CC)$, 
so here the initial characteristic-$0$ assumption of the paper is relevant).

\begin{lemm} \label{lemma-splitplane}
Let $\ell \subset \XXp$ be a line of the second kind. 
\begin{enumerate}
\item[a)] If a point $P$ appears in  $R_{\ell}$ with multiplicity one or two,
then the tangent plane $\operatorname{T}_{P}\XXp$ is a component of the ruled surface $\Ruledeight$. 
\item[b)] 
If  $P$ appears in  $R_{\ell}$ with multiplicity one, then the surface residual to  $\operatorname{T}_{P}\XXp$ in $\Ruledeight$
contains all lines on $\XXp$ that run through $P$.
\item[c)] If the support $\mbox{supp}(R_{\ell})$ consists of at least five points and 
a fiber $F_{P_2}$ contains two ramification points, then 
$\operatorname{T}_{P_2}\XXp$ appears with multiplicity at least two in $\Ruledeight$.
\end{enumerate}
\end{lemm}

\begin{proof} a) Let $F_P$ be a fiber of \eqref{eq:fibration-deg-d}.
According to \cite[$\S$~6]{Segre}, if 
 $\ii(\tilde{\ell}, F_P;P) \geq 2$, then either $P \in \sing(F_P)$
or $\ii(\tilde{\ell}, F_P;P) \geq 4$. The latter is ruled out by the assumption on multiplicity, so the claim of the lemma follows
directly from the definition of the family
${\mathfrak S}_{\ell}$. 

b) We maintain the notation of the proof of Lemma~\ref{lem-1line} and assume that the line $\ell'$ runs through $P$ on $\XXp$.
  One can easily check that for each point $\tilde{P}$  from a (punctured) neighbourhood
of $P$ on $\ell$ (in the complex analytic topology) the hessian quadric $V_{\tilde{P}}$ meets the tangent plane $T_{\tilde{P}} \XXp$ properly (i.e. along two lines: $\ell$ and $\ell_{\tilde{P}}$). Moreover, the line $\ell'$
lies in the closure of the union of the lines $\ell_{\tilde{P}}$.  

c) The claim results from an elementary but tedious computation.  We maintain the notation of the proof of Lemma~\ref{lem-1line}, assume that the points   $P := (1:0:0:0)$, $P_2:= (0:1:0:0)$ 
belong to $\mbox{supp}(R_{\ell})$ 
and that the fiber $F_{P}$ (resp.  $F_{P_2}$) is the quartic curve residual to $\ell$ in the intersection of $\XXp$ with the plane $V(x_3)$ (resp. $V(x_4)$).

At first we assume that $P$, $P_2$ are simple ramification points and $P_3 := (1:1:0:0) \in \mbox{supp}(R_{\ell}) \cap F_{P_2}$ (i.e. $c_1 = c_2 = 1$, $d_1 \neq  d_2$ and $d_1 d_2 \neq 0$ in \eqref{eq-lem41}).
Then we apply \eqref{eq-rem2-3-condition}  to compute the map \eqref{eq-paraRl}, 
and we find  the defining polynomial $h_{11}$ of $\Ruledeight$ 
 using the explicit resultant approach from the proof of Lemma \ref{lemma-ruledoctic}.
Finally, we check that $h_{11}$ is divisible by $x_4^2$, 
so the claim follows for the ramification types $(1^4, [1,1])$, $(1^2, [1,1]^2)$, $(2,1^2,[1,1])$.

To deal with the ramification type  $([1,1]^3)$ we assume   $c_1 = c_2 = 1$  and $0\neq d_1 = d_2$ in \eqref{eq-lem41}. 
Using the discriminant one can check that 
 $([1,1]^3)$  occurs only for $d_1 = 2$, whereas the ramification type  $(2,[1,1]^2)$ is impossible.  Then for $d_1 = 2$  we repeat the reasoning we used to deal with the other cases,
to complete the proof of the claim.
\end{proof}

%In Cor.~\ref{cor-type111111}  we should exclude the degenerations (i.e. two ramification points in the same fiber), 
%but an elementary computation shows that if a fiber $F_P$ contains two ramification points, then $\operatorname{T}_{P}\XXp$ appears with multiplicity at least two in $\Ruledeight$ (contrary to Segre's analytic approach in the proof of \cite[$\S$~6.]{Segre},  our direct argument is based on finding the map 
%\eqref{eq-paraRl} for the ramification type $(1^4,[1,1])$ (resp. $(1^2,[1,1]^2)$, resp. $([1,1]^3)$) 
%and using the explicit resultant approach from the proof of Lemma \ref{lemma-ruledoctic}, 
%so it works for all required characteristics, but only for quintics).

Lemma \ref{lemma-splitplane} has an immediate consequence for the valency of a line $\ell$ on a smooth quintic $X_5$
granted that all ramification points (relative to $\ell$) are simple:

\begin{cor} \label{cor-type111111} 
If $\ell \subset \XXp$ is a line of the second kind of ramification type $(1^6)$,
then $\ell$ is met by at most 24 other lines on $\XXp$:
\[
v(\ell)\leq 24.
\]
\end{cor}

\begin{proof}
By Lemma~\ref{lemma-splitplane}  all lines that meet $\ell$ on $\XXp$  lie on a degree-$5$ surface contained in $\Ruledeight$
(residual to the six distinct tangent planes at the ramification points).
The bound then follows straight from Bezout's theorem.
\end{proof}

In the remaining case of ramification type $(2,1^4)$,
we shall see that Segre's surface of principal lines still yields a bound that  is sufficient for our purposes.
However, one has to carry out 
a more detailed analysis of its behaviour. 
For this purpose, we put $\Ruledeight^r$  to denote the surface residual in  $\Ruledeight$ to 
the union of tangent planes of $\XXp$ in ramification points.

\begin{lemm} \label{lemma-splitplane-21}
Let $\ell \subset \XXp$ be a line of the second kind  of ramification type $(2,1^4)$ and let
$P$ appear in  $R_{\ell}$ with multiplicity two. 
\begin{enumerate}
\item[a)] 
If there are % (strictly) 
more than two lines through $P$ on $\XXp$, then $$\deg(\Ruledeight^r) \leq 5.$$ 
\item[b)]
 If % there are at most two lines through $P$ on $\XXp$ and 
$\deg(\Ruledeight^r) = 6$, then $\Ruledeight^r$ and $\XXp$ meet along the line $\ell$ with 
multiplicity $\geq 3$.
\end{enumerate}
\end{lemm}

\begin{proof} We maintain the notation of the proof of Lemma~\ref{lem-1line}, put $d_1=0$ 
(i.e. $P=(1:0:0:0)$ is the double ramification point),
and apply \eqref{eq-rem2-3-condition}  to find the parametrization \eqref{eq-paraRl}. 
We use the resultant to compute the polynomial $h_{11}$ that defines Segre's surface $\Ruledeight$ %$\Ruledeight$
(as indicated in the proofs of Lemmata \ref{lemma-ruledoctic}, \ref{lemma-splitplane}).
It is divisible by $x_3 x_4$,
so set
\[
h_9 = h_{11}/(x_3x_4) \;\;  \text{ and define the auxiliary surface } \;\; S_9 = V(h_9)\supset S_{11}^r.
\]
%To simplify our notation we put $\Ruledeight$ to denote the polynomial that defines Segre's surface $\Ruledeight$.

a)  The assumption on the number of lines through $P$ implies that $\alpha_{0221}$ vanishes (cf.\ Remark~\ref{rem-double-point}.a). 
By direct check, the polynomial $h_9$ is divisible by $x_3$ (resp.\ by $x_3 x_4$  when $V(x_4)$ contains two ramification points).
The claim thus follows from  Lemma~\ref{lemma-splitplane}.a. 

b) In general, 
one checks that the polynomial  $h_9$ belongs to the fifth power of the ideal $\langle x_3, x_4\rangle$. 
Thus, if 
 $\Ruledeight^r$ is a sextic, then it is singular along the line $\ell$ (this gives intersection multiplicity at least two). 
 The claim will be proved 
 by studying    
the intersection number of appropriate hyperplane sections of $\XXp$ and $S_9$.
%$\Ruledeight/(x_3 x_4)$  we arrive at the claim 
In detail, we cut $\XXp$ and $S_9$ with  the hyperplane $H_t:=V(x_1 - t x_2)$ that meets  $\ell$  transversally in the point $P'=(t:1:0:0)$). By direct check, the 
line $T_{P'}\XXp \cap H_t $ is always a component of the tangent cone of the curve $H_t \cap S_9$,
so the intersection multiplicity of $\XXp$ and $S_9$ along $\ell$ does exceed five. 
Since $\Ruledeight^r$ is a sextic by assumption,  the hypersurface  $S_9$ contains exactly three planes, 
each of which meets $\XXp$ along $\ell$ with multiplicity one. Thus the claim of Lemma~\ref{lemma-splitplane}.b follows.
\end{proof}

It should be pointed out that 
for a double ramification point
 the statement of Lemma~\ref{lemma-splitplane}.b does not carry over
(i.e. sometimes
the surface $S_{11}^r$  does not contain all lines through the point in question). 
 Yet Bezout's theorem
 (using Lemma \ref{lemma-splitplane-21}) immediately gives the following for ramification type $(2,1^4)$.

\begin{cor} \label{cor-type21111} 
If $\ell \subset \XXp$ is a line of the second kind of  ramification type $(2,1^4)$,
then $\ell$ is met by at most 28 other lines on $\XXp$:
\[
v(\ell)\leq 28.
\]
\end{cor} 

For completeness, we emphasize that  together with Corollary \ref{cor-type111111} this covers the two cases
missing from Section \ref{sect-sec-kind-quintics} for the proof of Theorem \ref{thm1}.

\begin{rem} 
\label{rem:deg9}
As explored for quartics in \cite{rs-advgeom},
one can construct examples of quintics with lines of other ramification types and   $\Ruledeight^r$ of  degree up to $9$. 
For such surfaces 
the approach we just followed yields weaker bounds for the valency than those needed for the proof of Theorem \ref{thm1}.  
In the next section we will overcome this subtlety using the results of  Sect.~\ref{sect-segre-functions}.
\end{rem}

\section{Counting lines along a line on a quintic surface}

In this section we can finally come to the problem that was our original motivation. 
We complete the proof of Theorem \ref{thm1} and  show that the given bound for the valency is in fact sharp 
(Example~\ref{example-3squared-28}).
%
%\begin{theo} \label{thm-28}
%A line on  a smooth quintic surface $\XXp$ is met by at most $28$ other lines on $\XXp$.
%\end{theo}
Before the proof, let us formulate an immediate consequence of Theorem~\ref{thm1}
which already points towards Theorem \ref{thm2}:

\begin{cor} \label{cor-125}
If a smooth quintic surface contains five coplanar lines, then it contains at most $125$ lines. 
\end{cor}

\begin{proof}[Proof of Corollary \ref{cor-125}]
Any line in $\PP^3$ meets the hyperplane which decomposes into the 5 given lines when intersected with the quintic $X_5$.
Thanks to Theorem \ref{thm1}, the total number of lines on $X_5$ is thus bounded by $5+5\cdot(28-4) = 125$ as claimed.
\end{proof}

\subsection{Proof of Theorem~\ref{thm1}}
Let $\ell \subset \XXp$ be a  line. 
By Prop.~\ref{prop-sfk} and Corollaries~\ref{cor-type111111},~\ref{cor-type21111},  we can assume that $\ell$ is of the second kind and 
$\mbox{supp}(R_{\ell})$ consists of at most four points (as given in \eqref{eq-rrl-five}). 
The proof is based on the case-by-case study of the $\sss_k$-functions from Def.~\ref{def-bf} 
for each possible ramification type. 

Without loss of generality we assume that the quintic surface $\XXp$ is given by 
% the polynomial 
\eqref{eq-expansion-d}
% where
and  $\ell = V(x_3, x_4)$. 

Recall that Lemma \ref{lemm-lines-second-kind} implies that  the rational function  $\sss_0$ vanishes identically.
For $j=1,2$ we define the polynomials $\tilde{\sss}_j$, $\ddd_j$ $\in \kk[x_2]$  by the formula 
\eqref{eq-fraction-decomp}.
In order to apply Prop.~\ref{lemm-fundamental},
we have to check the condition $\partial f/\partial x_4\neq 0$ on $\ell$.
Presently, one finds that
\[
V\left(\frac{\partial f}{\partial x_4}\right) \cap \ell = F_{P_1} \cap \ell \subset V(x_3).
\]
Hence all the lines on $\XXp$ intersecting $\ell$ are either
\begin{itemize}
\item
contained in the plane $V(x_3)$ or
\item
$\ell$-ramified outside the plane $V(x_3)$ 
(so that restrictions imposed by the index of ramification apply, including Lemma \ref{lem-1line})
or
\item
$\ell$-unramified outside the plane $V(x_3)$
(so that the total number is bounded by both $\deg(\tilde{\sss}_1)$ and $\deg(\tilde{\sss}_2)$).
\end{itemize}
In what follows, we will balance out these three cases against each other
in order to prove the valency bound from Theorem \ref{thm1}.
%
%we infer that the number of $\ell$-unramified lines which are not contained in the fiber $F_{P_1}$ 
%(i.e. in the plane $V(x_3)$ - because in our set-up we have $\textstyle{\frac{\partial f}{\partial x_4}}(P_1) = 0$)
%is bounded by both $\deg(\tilde{\sss}_1)$ and $\deg(\tilde{\sss}_2)$.   
%

\subsection{Triple ramification points}
\label{ss:triple}

At first we study the ramification types with a triple point and assume that  $\XXp$ lies in the image of  
${\mathcal W} \times \kk^{16}$ under the parametrization \eqref{eq-paraRl} where ${\mathcal W} \subset \kk^6$ is one of the affine subspaces
we discussed in \ref{example-one-triple-point} and \ref{example-one-triple-point-b}.

\subsubsection{Ramification type $(3^2)$} We maintain the notation of \ref{example-one-triple-point}. Observe that the fiber $F_{P_1}$ 
meets $\ell$ only in % the point 
$P_1$, so $F_{P_1}$  contains no $\ell$-unramified lines.
On the other hand, all $\ell$-ramified lines (if any) are contained in the fibres $F_{P_1}, F_{P_2}$.
 
Suppose that $\sss_1 \not\equiv 0$. Since, by direct computation, we have  
$$\deg(\tilde{\sss}_1) \leq 17,$$ 
and there are at most four  lines in the fibers $F_{P_1}$ (resp. $F_{P_2}$), 
Prop.~\ref{lemm-fundamental} yields that 
the line  $\ell$ is met by at most 25 other lines on $\XXp$.

Assume that $\sss_1 \equiv 0$.
We have $\deg(\tilde{\sss}_2) \leq 28$, so the valency bound stated in Theorem \ref{thm1}  
follows directly from  Prop.~\ref{lemm-fundamental} 
provided the fibers $F_{P_1}$, $F_{P_2}$ contain no lines. 
To complete the proof in this case, it remains to study $\tilde{\sss}_2$ more precisely. 
We have (with short-hand notation t.o.d. = terms of degree)
$$
\tilde{\sss}_2 = \alpha_{2021} \cdot (\mbox{t.o.d. } 28, 26)  +  (\mbox{t.o.d. }24, \hdots, 4) + \alpha_{0212} \cdot (\mbox{t.o.d. } 3, 0) 
$$
Therefore, if one of the fibers $F_{P_1}$, $F_{P_2}$ 
contains a line $\tilde{\ell} \neq \ell$, then by Remark~\ref{rem-triple-points} and 
 Prop.~\ref{lemm-fundamental},
we have  $\alpha_{0212}=0$ or $ \alpha_{2021}=0$, and the line $\ell$ is met by at most $24$
$\ell$-unramified lines (resp.~at most $20$ if both fibers contain a line). 
This completes the proof of Theorem \ref{thm1} for ramification type $(3^2)$. % is complete. 

%Therefore, if one of the fibers $F_{P_1}$, $F_{P_2}$ (resp. both) contains a line $\tilde{\ell} \neq \ell$, then by Remark~\ref{rem-triple-points} and 
% Prop.~\ref{lemm-fundamental},
%either  $\alpha_{0212}=0$ or $ \alpha_{2021}=0$ (resp. both $\alpha_{0212}$ and $\alpha_{2021}$ vanish), and the line $\ell$ is met by at most $24$ (resp. $20$)
%$\ell$-unramified lines. 
%This completes the proof of Theorem \ref{thm1} for ramification type $(3^2)$. % is complete. 
%\hfill $\Box_{(3^2)}$

\subsubsection{Ramification type $(3,2,1)$} 
We consider the case $c=0$ in \ref{example-one-triple-point-b} and check that the denominator of 
$\sss_j$ divides the product $(x_1^4 (x_1-x_2)^5 x_2^{19})$ for   $j=1,2$
 (here the factor $(x_1-x_2)^5$ comes from the vanishing of the partial 
$\partial f/\partial x_4$, whereas the factor given by the simple ramification point $P_3$ 
(cf. Remark~\ref{rem-triple-point}.a) cancels out).

Assume $\sss_1 \not\equiv 0$. 
One can easily check that  $\deg(\tilde{\sss}_1) \leq 20$. Moreover, if 
$F_{P_1}$ contains an $\ell$-unramified line (i.e. a line through $(1:1:0:0)$)
%there is an extra line through $(1:1:0:0)$, 
then
$\deg(\tilde{\sss}_1)$ cannot exceed $19$.
Since the fiber  $F_{P_1}$ (resp. $F_{P_2}$) contains at most three 
(resp. four) lines through the point $P_1$ (resp.~$P_2$),
 Prop.~\ref{lemm-fundamental} combined with Lemma \ref{lem-1line} and Remark~\ref{rem-triple-point}.a
implies that $\ell$ is met by at most $28$ other lines on $\XXp$.  

If  $\sss_1 \equiv 0$, then $\alpha_{0221} = 0$  and the degree  of $\tilde{\sss}_2$ is at most $24$. %, because  $\alpha_{0221}$ vanishes then. 
Moreover, vanishing of  $\alpha_{2021}$, $\alpha_{4001}$ results in  $\deg(\tilde{\sss}_2) \leq 20$, whereas the existence of 
% a line through $(1:1:0:0)$ 
an $\ell$-unramified line in $F_{P_1}$
diminishes $\deg(\tilde{\sss}_2)$ by one. The required bound 
thus results from  Prop.~\ref{lemm-fundamental} 
combined with Remark~\ref{rem-triple-point}.a.  %\hfill $\Box_{(3,2,1)}$

\subsubsection{Ramification type $(3,1^3)$} Again, we maintain the notation of \ref{example-one-triple-point-b}
and fix any $c\in\kk$ that does not result in the ramification type $(3,2,1)$. We obtain
% $$ the dj looks in the following way - we do not need it
% d_j |  (x_1^4 (x_1-x_2)^5 (c x_1-x_2)^5 x_2^{10})    \mbox{ for }  j=1,2
% $$
$$
\deg(\tilde{\sss}_1) \leq 17  \;\;\; \mbox{ and }  \;\; \deg(\tilde{\sss}_2) \leq 24. 
$$
If $\tilde{\sss}_1 \not\equiv 0$, then the claim follows directly from Lemma~\ref{lem-1line} (indeed, we have at most eight lines in the fibers $F_{P_1}$, $F_{P_2}$ and at most one line through each of the other two simple ramification points). %  Remark~\ref{rem-triple-point}.c.
Otherwise, we use Remark~\ref{rem-triple-point}.b to check that   $\deg(\tilde{\sss}_2) \leq 20$
whenever $F_{P_2}$ contains  a line through % the point 
$P_2$.  
%$\mbox{}$ \hfill $\Box_{(3,1^3)}$

\subsection{Double ramification points}
It remains to deal with the ramification types with two double points. 
As in \ref{ss:triple}, we
assume that  $\XXp$ lies in the image of  
${\mathcal W} \times \kk^{16}$ under the parametrization \eqref{eq-paraRl}, where ${\mathcal W} \subset \kk^6$ is one of the affine subspaces
we discussed in \ref{example-double-point}. 

We maintain the notation of \ref{example-double-point}. Observe that
the partial derivative $\partial f/\partial x_4$ vanishes in $P_1$  and in the point $(1:c:0:0)\in F_{P_1}$
(thus the latter is the only unramified point where the first assumption of Prop.~\ref{lemm-fundamental} is not fulfilled
(but the second is)) . 

%%%%%%%%%%%%%%%%%%%%%%%%%%%%%%%%%%%%%%%%%%%%%%%%%%%%%%%%%%%%%%%%%%%%%%%%%%%%%%%%%%%%%%%%%%%%%%%%%%%%%%%%%%%%%%%%%%%%%%%%%%%%%%%%%%%%%
%%%%%%%%%%%%%%%%%%%%%%%%%%%%%%%%%%%%%%%% ramification 2^3  
%%%%%%%%%%%%%%%%%%%%%%%%%%%%%%%%%%%%%%%%%%%%%%%%%%%%%%%%%%%%%%%%%%%%%%%%%%%%%%%%%%%%%%%%%%%%%%%%%%%%%%%%%%%%%%%%%%%%%%%%%%%%%%%%%%%%%

\subsubsection{Ramification type $(2^3)$} \label{subsec-twothree} We assume $c=4$ in the formulas from \ref{example-double-point}. 
 
Suppose that $\sss_1 \not\equiv 0$. 
By direct computation we have $\deg(\tilde{\sss}_1) \leq 26$ (resp. $\deg(\tilde{\sss}_1) \leq 25$ 
when there is an extra  line through $(1:4:0:0)$).
Moreover, if there are at least two extra lines through $P_2$, then Remark~\ref{rem-double-point}.a 
yields $\alpha_{2012}=0$ and one gets $\deg(\tilde{\sss}_1) \leq 23$. More generally, one can check
that if  there are at least  two lines $\neq \ell$  through each of $k$ ramification points,  than 
 $\deg(\tilde{\sss}_1) \leq 26 - 3k$.  
Consequently, Prop.~\ref{lemm-fundamental} yields the required bound
for the valency from Theorem \ref{thm1} provided  $\sss_1 \not\equiv 0$ and either $\ell$ is the unique line through the ramification point
$P_1$ or there are at least three lines on $\XXp$ which  contain the point $P_1$.

It remains to discuss the case when there are exactly $26$ $\ell$-unramified lines  and exactly one line $\neq \ell$ runs through each  point $P_j$ for $j = 1, \ldots, 3$. 
By  Prop.~\ref{lemm-fundamental} every root of $\tilde{\sss}_1$ must be a root of   $\tilde{\sss}_2$, so the remainder of
the division of  $\tilde{\sss}_2$ by  $\tilde{\sss}_1$ vanishes. An elementary (but tedious) computation
yields $\alpha_{2012}=0$ and  at most $23$ $\ell$-unramified lines. Contradiction.

Finally suppose that $\sss_1 \equiv 0$. Then one can easily see that  $\alpha_{2012}$, $\alpha_{0221}$,  $\alpha_{4001}$, $\alpha_{4010}$,  $\alpha_{0410}$, $\alpha_{0401}$ vanish. By Remark~\ref{rem-double-point}.b
none of the fibers  $F_{P_1}$, $F_{P_2}$ contains a line. Moreover, we have
 $\deg(\tilde{\sss}_2) \leq 24$, so this completes the proof of Theorem \ref{thm1} for ramification type $(2^3)$.

\subsubsection{Ramification type $(2^2,1^2)$}
 We assume that $c \neq 0,1,4,-8$ in \ref{example-double-point} % \marginpar{$c\neq 1$ included} c=1 gives a singular quintic
 and put  $P_3$, $P_4$ to denote the ramification points of index $1$ (they can be found with the help of $\mathfrak{d}_2(f)$). Recall that, by Lemma~\ref{lem-1line}, there is at most one line $\neq \ell$ through $P_3$ (resp. $P_4$). 
Moreover, one can easily check that once there is a line (different from $\ell$) on $\XXp$ through $(1:c:0:0)$,
the degree of $\tilde{\sss}_1$ drops by one.

Suppose that $\sss_1 \not\equiv 0$. We have  $\deg(\tilde{\sss}_1) \leq 23$. 
If neither $\alpha_{2012}$ nor $\alpha_{0221}$ vanishes, the bound $v\leq 28$ follows from 
Remark~\ref{rem-double-point}.a. 
In the remaining cases, a direct computation shows that if  $\alpha_{2012} = 0$ (resp.  $\alpha_{0221} = 0$), 
then we have $\deg(\tilde{\sss}_1) \leq 20$ (resp. $\tilde{\sss}_1$ is divisible by 
$x_2^3$). Prop.~\ref{lemm-fundamental} thus yields the claim again.

Eventually we are led to assume $\sss_1 \equiv 0$. Then  $\alpha_{2012}$, $\alpha_{0221}$,  $\alpha_{4001}$, $\alpha_{4010}$,  $\alpha_{0410}$, $\alpha_{0401}$ vanish and we have
$\deg(\tilde{\sss}_2) \leq 20$. 
Remark~\ref{rem-double-point}.b
yields that there are no lines $\neq \ell$ through $P_1$, $P_2$, which completes the proof of Theorem \ref{thm1}
for non-degenerate ramification type  $(2^2,1^2)$.
The proof for the degenerate type follows the same lines, so we omit it for the sake of brevity. 
\qed

\subsection{Record valency}

In the above part of the proof of Thm~\ref{thm1} we did not analyze (very precisely) the behaviour of $\sss_2$ under the assumption that $\sss_1$ vanishes along $\ell$.
Such an analysis (the details of which we omit) yields the following corollary.

\begin{cor} \label{cor-28abc}
Let $\XXp \subset \PP^3$ be a smooth quintic surface and let $\ell \subset \XXp$ be a  line. If $\ell$ is met by 
$28$ other lines on $\XXp$, then one of the following holds:
\begin{itemize}
\item[$(28_a)$]
the line $\ell$ is of the type $(3^2)$ and $\sss_1$ vanishes along $\ell$, 
\item[$(28_b)$]  the line $\ell$ is of the type $(2^3)$, and $\sss_1$ does not vanish along  $\ell$,
\item[$(28_c)$]  the line $\ell$ is of the type $(2,1^4)$.
\end{itemize}
\end{cor}

Finite field experiments suggest that the configurations ($28_b$), ($28_c$) might not exist, 
but proving this conjecture exceeds the scope of this paper.   
Below we describe the family of all quintics that carry the configuration ($28_a$).
In particular this shows that Thm~\ref{thm1} is  sharp.

\begin{example}[Quintics of type  $(28_a)$] 
\label{example-3squared-28}
Let us maintain the notation of \ref{example-one-triple-point} and assume that $\XXp$,
endowed with a line with two triple ramification points, satisfies the conditions 
\eqref{eq-quintic33-sc}.
We solve the system of equations given by the vanishing of $\sss_1$ along the line $\ell$. 
An elementary computation yields that $\alpha_{4001}$, $\alpha_{0410}$ vanish, we have
\begin{eqnarray*}
\alpha_{4010}=\textstyle{-\frac{1}{16}}\,{\alpha_{2021}}^{2}\alpha_{1112} , && \alpha_{{0401}}=\textstyle{-\frac{1}{16}}\,\alpha_{1121} \alpha_{{0212}}^{2}, \\
\alpha_{{3101}}=\textstyle{-\frac{1}{16}}\,\alpha_{2021}{\alpha_{{1112}}}^{2} ,  && \alpha_{{1310}}=\textstyle{-\frac{1}{16}}\,{\alpha_{{1121}}}^{2}\alpha_{{0212}} ,
\end{eqnarray*}
and the following equalities hold
\begin{eqnarray*}
\alpha_{{3110}} &=&
- \textstyle{\frac{1}{8}} \,  \alpha_{{2021}}\alpha_{{1112}}\alpha_{{1121}} - \textstyle{\frac{1}{16}} \,{\alpha_{2021}}^{2}\alpha_{{0212}}   , \\
\alpha_{{2201}}&=&
- \textstyle{\frac{1}{8}}\,\alpha_{{2021}}\alpha_{{1112}}\alpha_{{0212}} - \textstyle{\frac{1}{16}} \,{\alpha_{{1112}}}^{2}\alpha_{{1121}}, \\
\alpha_{{2210}}&=&
-\textstyle{\frac{1}{8}} \,\alpha_{{2021}}\alpha_{{1121}}\alpha_{{0212}}-  \textstyle{\frac{1}{16}} \,\alpha_{{1112}}
{\alpha_{{1121}}}^{2}, \\
 \alpha_{{1301}}&=&  -  \textstyle{\frac{1}{8}} \,\alpha_{{1112}}\alpha_{{1121}}\alpha_{{0212}}  -  \textstyle{\frac{1}{16}} \,\alpha_{{2021}}{\alpha_{{0212}}}^{2}  \, .
 \end{eqnarray*}
In this way we may define a map
\begin{equation} \label{eq-paraRl-z5}
{\mathcal Z}_5 \, : \, \kk^4 \times \kk^{6} \rightarrow  {\mathcal O}_{\PP^3}(5).
\end{equation}
%By abuse of notation we will use ${\mathcal Z}_5$ to denote its image. 
Consider the image of $\mathcal Z_5$, a family of quintic surfaces.
A standard Groebner basis computation shows that the quintic surface 
$Z={\mathcal Z}_5(1, \hdots, 1)$ is smooth. By Remark~\ref{rem-triple-points}  none of the (ramified) fibers $F_{P_1}$, $F_{P_2}$ contain a line.   %%%%%%%%%%%%%%%%%%%%%%%%tbbb
Moreover, for the line $\ell := V(x_3, x_4)$ on $Z$  we have $\sss_1 = 0$, $\deg(\tilde{\sss}_2)=28$ and
the discriminant of $\tilde{\sss}_2$ does not vanish. 
%\marginpar{which characteristics?}
By Prop.~\ref{lemm-fundamental}, the line $\ell$ meets exactly $28$ other lines on $Z$.
\end{example}
%and  the bound of Thm~\ref{thm1} is thus sharp.   

\begin{rem}
As we will see in Section \ref{bounding-lines-quintics} (by Lemma~\ref{lemma-colines}), 
if a line on a smooth quintic surface $\XXp$ is met by at least $26$ other lines, then
$\XXp$  contains a fivetuplet of  coplanar lines (including the given line).  
Thus no smooth surface in the family ${\mathcal Z}_5$ may contain more than $125$ lines by Corollary \ref{cor-125}.
\end{rem}

%%%%%%%%%%%%%%%%%%%%%%%%%%%%%%

\section{Lines vs. conics or twisted cubics}

Before we can come to the proof of Theorem \ref{thm2},
we need a few more preparations, but now in a slightly different direction.
Namely we want to get a rough idea of the number of lines on a quintic
meeting a given (smooth) rational curve $C$ -- here either an irreducible conic or a twisted cubic.
To this end, we start by following fairly closely the lines of \cite{rs-2014}.

At any point $P\in C$, we can compare the tangent plane to $X_5=V(f)\subset\PP^3$,
\[
T_P = T_P X_5,
\]
and the Hessian quadric
$V_P\subset\PP^3$ defined by \eqref{eq:q_P}.
%by the quadratic form $q\in \kk[x_1,x_2,x_3,x_4]$ obtained 
%from the Hessian matrix associated to $f$:
%\[
%q_P(x_1,x_2,x_3,x_4) =  \frac 12\, (x_1,x_2,x_3,x_4) \left(  \frac{\partial^2 f}{\partial x_i \partial x_j}(P)\right)_{1\leq i,j\leq 4} \phantom{l}^t(x_1,x_2,x_3,x_4).
%\]
The intersection
$
T_P \cap V_P
$
generically consists of the two 3-contact lines of $X_5$ at $P$;
in particular, if some line $\ell\subset X_5$ contains $P$,
then
\begin{eqnarray*}
\label{eq:V_P}
\ell\subset T_P\cap V_P.
\end{eqnarray*}
Depending on the chosen curve $C$, we now turn to the total space
\begin{eqnarray}
\label{eq:Z}
Z = Z_C = \cup_{P\in C} (T_P\cap V_P) \subset \PP^3.
\end{eqnarray}

\subsection{Conic case}

\begin{prop}
\label{prop:Z}
Assume that $C$ is an irreducible conic. Then
$Z$ is a surface of degree $22$ in $\PP^3$.
It does not contain $X_5$,
but the surface $Z$ contains
any line in $X_5$ meeting $C$.
\end{prop}

\begin{proof}
The proof follows essentially word by word that of \cite[Prop.~2.2]{rs-2014}.
\end{proof}

\begin{cor}
\label{cor88}
A conic on a smooth quintic surface $X_5\subset\PP^3$ meets at most $88$ lines.
\end{cor}

\begin{proof}
If the conic $C$ is geometrically reducible, then the claim follows from Theorem \ref{thm1}.
Else consider the effective divisor 
\[
D = Z \cap X_5 \in |\mathcal O_{X_5}(22)|
\]
and let $m$ denote the multiplicity of $C$ in $D$.
Then, since all lines meeting $C$ are contained in $D$,
we obtain two inequalities
\begin{eqnarray}
\label{eq:88}
\#\{\text{lines on $X_5$ meeting } C\} \leq
\begin{cases}
\deg(D-mC) = 110-2m,\\
C.(D-mC) = 44+4m.
\end{cases}
\end{eqnarray}
The estimate in Corollary \ref{cor88} arises from $m=11$.
\end{proof}

\begin{rem}
\label{rem88}
If there are lines on $X_5$ planar with the conic $C$,
then the estimate for total number of lines meeting $C$ can be improved.
What's more important for us:
the bound 
\[
C.\left(\sum_{i=1}^M \ell_i\right) \leq 88.
\]
still holds where the second divisor comprises (all) lines contained in $X_5$.
\end{rem}

Corollary \ref{cor88} and Remark \ref{rem88}
will prove quite useful in the proof of Theorem \ref{thm2}
as they facilitate simplifications of  some of the arguments involved. % in the next section.

\subsection{Twisted cubic case}
\label{ss:tcc}

Twisted cubic curves show a rather different behavior than conics -- notably because they are not plane.
We first investigate the surface $Z$ from \eqref{eq:Z}
arising from the twisted cubic $C$.

\begin{prop}
\label{prop:twisted_cubic}
Assume that $C$ is a twisted cubic on a smooth quintic $X_5\subset\PP^3$. Then
$Z$ is a surface of degree $33$ in $\PP^3$.
It does not contain $X_5$,
but the surface $Z$ contains
any line in $X_5$ meeting $C$.
\end{prop}

\begin{proof}
By a linear transformation, we can assume $C$ to be parametrized by
\begin{eqnarray*}
\varphi: \;\; \PP^1\; & \to & \;\;\;\;\;\;\;\;\;C\\
~[s,t] & \mapsto  & [t^3,s^3,t^2s,ts^2]
\end{eqnarray*}
The homogeneous ideal of $C$  is thus generated by the three quadrics
\[
q_1 = x_1x_2-x_3x_4,\; \;\; q_2 = x_1x_4-x_3^2, \;\;\; q_3 = x_2x_3-x_4^2.
\]
At the point $P=\varphi(s,t)$, the tangent plane to the quintic $X_5$
is thus given by a bihomogeneous polynomial 
\[
g\in \KK[s,t][x_1,\hdots,x_4]
\]
which is linear in the $x_i$ and of degree $12$ in $s,t$.
Meanwhile the quadratic form $q$ defining $V_P$ has degree $9$ in $s,t$.
$Z$ is defined by the resultant of $g$ and $q$ with respect to $t$
which thus has degree $33 = 2\cdot 12 + 1\cdot 9$.

To conclude, we note as in \cite{rs-2014} that any component of $Z$ is automatically ruled by lines.
In particular, neither component  can equal $X_5$.
\end{proof}

Contrary to Corollary \ref{cor88}, Proposition \ref{prop:twisted_cubic} alone does not give a serious estimate
for the number of  lines on $X_5$ intersecting the twisted cubic $C$.
Arguing as in the proof of Corollary \ref{cor88} with the divisor
\begin{eqnarray}
\label{eq:DC}
D = Z\cap X_5 \in |\mathcal O_{X_5}(33)|,
\end{eqnarray}
we see at least that a plentitude of adjacent lines
forces $C$ to have large multiplicity in $D$.
The combination with the following result 
will prove quite useful in the sequel.

\begin{prop}
\label{prop:double}
Assume that some line $\ell$ on the smooth quintic $X_5\subset\PP^3$
meets the twisted cubic $C$ with multiplicity two.
Then $\ell$ has multiplicity at least two in $D$,
i.e. $D-2\ell>0$.
\end{prop}

\begin{proof}
Assume that $\ell$ and $C$ intersect in two distinct point $P_1, P_2$.
On top of the above normalization,
we can assume that 
\[
I(\ell)=(x_3,x_4) \;\;\; \text{ and } \;\;\; 
P_1 = [1,0,0,0], \;\;\; P_2 = [0,1,0,0].
\]
By assumption, $f\in I(C) \cap I(\ell)$, i.e. there are cubic polynomials $c_i\in\KK[x_1,\hdots,x_4]$ such that
\[
f = q_1 c_1 + q_2 c_2 + q_3 c_3 \;\;\; \text{ with } \;\;\; c_1\in (x_3,x_4).
\]
In the end, it will suffice to consider polynomials modulo $st$
(such that only the pure powers in $s$ and in $t$ remain);
for this purpose, the following observation concerning the evaluation at $P=\varphi(s,t)$ will be useful:
\[
c_1, q_2, q_3 \in (x_3,x_4) \;  \Longrightarrow \; st \mid c_1(P), q_2(P), q_3(P).
\]
Moreover, the same holds true for the (iterated) partial derivates with respect to $x_1$ and $x_2$.
With this in mind, it is easy to get a hand on the defining polynomials $g$ of $T_P(S)$ and $q$ of $V_P$.
For instance,
\[
g \equiv s^3 c_3(P) x_3 + t^3 c_2(P) x_4 \mod (st).
\]
In other words, $g\in(x_3,x_4,st)$, and the same holds for $q$ 
(as one easily checks by  considering the double partials with respect to $x_1, x_2$).

We continue by computing the resultant of $g$ and $q$ via the Sylvester matrix with respect to $s$ or $t$.
The above argument shows that the first and the last column of this matrix has all (two) entries in $(x_3,x_4)$.
Hence the determinant of the matrix is in $(x_3,x_4)^2$,
and the line $\ell$ has multiplicity at least two in $D$ as claimed.

If the line $\ell$ meets $C$ tangentially, say at $[1,0,0,0]$,
then $I(\ell)=(x_2,x_4)$, and the analogous argument shows that
\[
g, q \in (x_2, x_4, s^2).
\]
Hence the first two (or the last two) columns of the Sylvester matrix have entries in $(x_2, x_4)$,
and we conclude as before.
\end{proof}

%%%%%%%%%%%%%%%%%%%%%%%%%%%%%%%%%%%%%%%%%%%%%%%%%%%%%%%%%%%%%%%%%%%%%%%%%%%%%%%%%%%%%%%%%%%%%%%%%%%%%%%%%%%%%%%%%%%%%%%%%%%%%%%%%%%%%%%%%%%%%%%%%%%%%%%%%%%%
%%%%%%%%%%%%%%%%%%%%%%%%%%%%%%%%%%%%%%%%%%%%%%%%%%%%%%%%%%%%%%%%%%%%%%%%%%%%%%%%%%%%%%%%%%%%%%%%%%%%%%%%%%%%%%%%%%%%%%%%%%%%%%%%%%%%%%%%%%%%%%%%%%%%%%%%%%%%
%%%%%%%%%%%%%%%%%%%%%%%%%%%section:  bounding the number of lines on quintics
%%%%%%%%%%%%%%%%%%%%%%%%%%%%%%%%%%%%%%%%%%%%%%%%%%%%%%%%%%%%%%%%%%%%%%%%%%%%%%%%%%%%%%%%%%%%%%%%%%%%%%%%%%%%%%%%%%%%%%%%%%%%%%%%%%%%%%%%%%%%%%%%%%%%%%%%%%%%%
%%%%%%%%%%%%%%%%%%%%%%%%%%%%%%%%%%%%%%%%%%%%%%%%%%%%%%%%%%%%%%%%%%%%%%%%%%%%%%%%%%%%%%%%%%%%%%%%%%%%%%%%%%%%%%%%%%%%%%%%%%%%%%%%%%%%%%%%%%%%%%%%%%%%%%%%%%%%%

\section{Bounding the total number of lines on quintics} 
\label{bounding-lines-quintics}

This section has a slightly different flavour than the previous sections  of the paper
as the methods are rather different.
Our aim is to complete the proof of our second main result, Theorem \ref{thm2}.
Recall from Corollary \ref{cor-125}
that if a smooth quintic surface $X_5\subset\PP^3$ admits a hyperplane section splitting into five lines, then  $\XXp$
contains at most $125$ lines. 
In order to prove Theorem \ref{thm2},
it thus remains to study the other possible configurations of lines.
%To this end, we introduce the following terminology:

%\begin{Definition}
%\label{def}
%If a line $\ell \subset \XXp$ is met by two other coplanar lines, but it is not met by four coplanar lines, we call it
%an {\sl s-line}. We call $\ell$  a {\sl c-line} if it is met by at least one other line 
%on $\XXp$, but it is not an s-line (neither does it meet four coplanar lines). We call it an {\sl o-line} when it meets 
%no other lines on the quintic surface. 
%
%A triplet of coplanar lines is called a {\sl triangle}. 
%A conic (resp.~a planar cubic) residual to a  triangle (resp.~a pair of coplanar lines) is referred to as
% {\sl the conic} (resp.~{\sl the cubic}) {\sl of the triangle} (resp.~{\sl of the pair}).
%\end{Definition}

% Combined with 
%Cor.~\ref{cor-125} this yields the following proposition
%\begin{prop} \label{prop-129lines-general} 
%There are at  most 129  lines on a smooth quintic surface. % $\XXp$.
%\end{prop}

Recall  that the locus of points $P$ such that there exists a line that meets $\XXp$ with multiplicity at least $4$ in $P$ 
is the support of a divisor in ${\mathcal O}_{\XXp}(31)$ 
%(outside characteristics $2,3,5$, as excluded in Convention \ref{conv}, 
(see \cite{clebsch},  \cite[$\S$~8]{kollar}). 
%, and for the positive characteristic bit \cite{voloch}).
Following classical terminology, we  call it the \emph{flecnodal divisor} of $\XXp$ and denote it by  ${\mathcal F}$.  
Observe that each line on $\XXp$ is a component of ${\mathcal F}$.  
%\vspace*{1ex}
%We introduce the following notation: \\
%If a line $\ell \subset \XXp$ is met by two other coplanar lines, but it is not met by four coplanar lines, we call it
%an {\sl s-line}. We call $\ell$  a {\sl c-line} if it is met by at least one other line 
%on $\XXp$, but it is not an s-line. We call it an {\sl o-line} when it meets 
%no other lines on the quintic surface. \\
%A triplet of coplanar lines is called a {\sl triangle}. 
%A conic (resp. a planar cubic) residual to a  triangle (resp. a pair of coplanar lines) is {\sl the conic} (resp. {\sl the cubic}) {\sl of the triangle} (resp. {\sl of the pair}). \\
A curve on $\XXp$ is called {\sl flecnodal} if and only if it is a component of 
the flecnodal divisor  ${\mathcal F}$. 
% We call a line {\sl reduced} if and only if it comes up with multiplicity one in  ${\mathcal F}$. Otherwise it is a 
% {\sl non-reduced} line. 

We record the following simple observation. 

\begin{lemm} \label{lemma-colines}
\label{lem:25}
If a line $\ell\subset X_5$ is not contained in any hyperplane splitting into lines, then
%An s-line $\ell$ on a smooth quintic surface $\XXp$ is met by at most $25$ other lines:
\[
v(\ell)\leq 25.
\]
% If the number of lines in question is $25$, then $\ell$ is a component  of exactly $12$  triangles on $\XXp$. \\
%A c-line $\ell$ on $\XXp$ intersects at most $21$  other lines on the quintic surface:
%\[
%v(\ell)\leq 21.
%\]
\end{lemm}

\begin{proof}
This is a consequence of the Euler number formula \cite[Prop.~III.11.4]{bpvh}  for the fibration \eqref{eq:fibration-deg-d} and 
the Euler number computation for singular curves (see \cite[Cor.~V.4.4.ii]{dimca}). %p.162
\end{proof}

In order to start working towards the proof of Theorem \ref{thm2},
we make from now the following assumption:

\begin{assumption}
\label{ass}
The smooth quintic $X_5$ contains $M>127$ lines.
\end{assumption}

In particular, Corollary \ref{cor-125} implies that $X_5$ does not admit any hyperplane splitting into five lines.
It thus follows from Lemma \ref{lemma-colines}
that $v(\ell)\leq 25$ for any line on $X_5$.
Numbering  the lines on $X_5$ from $\ell_1$ to $\ell_M$,
we  introduce the following auxiliary effective divisor:
\begin{eqnarray}
\label{mF'}
\mF':= \mF- \sum_{i=1}^M \ell_i,\;\;\; \deg(\mF') = 155-M \leq 27.
\end{eqnarray}
%By assumption, $\mF'$ has degree 
%\[
%\deg(\mF') = 155-M\leq 29.
%\]
Note that we can compute the intersection number of $\mF'$ and any line $\ell\subset X_5$ as follows:
\begin{eqnarray}
\label{eq9}
\ell.\mF' = \ell.\left(\mF-\sum_{i=1}^M\ell_i\right) = 31 - (v(\ell) - 3) = 34-v(\ell) \geq 9
\end{eqnarray}
where the last estimate follows from Lemma \ref{lemma-colines}.

The overall idea of our approach goes as follows: 
For large $M$, the intersection number in \eqref{eq9} becomes too big relative to the degree of $\mF'$.
Indeed, for a single line, such a phenomenon might well occur, but not for all lines at the same time 
as we shall see soon.

For later reference, we compute the self-intersection number $(\mF')^2$ in two ways.
We know that $\mF^2 = 5\cdot 31^2$. Let us compare this intersection number with the decomposition \eqref{mF'}.
Clearly we have
\[
\left(\sum_{i=1}^M \ell_i\right)^2 = \sum_{i=1}^M \left(\ell_i.\sum_{j=1}^M \ell_j\right) = \sum_{i=1}^M (v(\ell_i)-3) = \sum_{i=1}^M v(\ell_i)-3M
\]
while 
\[
\sum_{i=1}^M \ell_i.\mF' = \sum_{i=1}^M (34-v(\ell_i)) = 34M - \sum_{i=1}^M v(\ell_i).
\]
Together this yields, using Lemma \ref{lemma-colines},
\begin{eqnarray}
\label{eqq2}
\mF^2 = (\mF')^2 + 65M  - \sum_{i=1}^M v(\ell_i)\geq (\mF')^2 + 40 M.
% 2 (34M - \sum_{i=1}^M v(\ell_i)) + \sum_{i=1}^M v(\ell_i) - 3M
\end{eqnarray}
Assumption \ref{ass} thus gives
\begin{eqnarray}
\label{eqq}
(\mF')^2 \leq 5\cdot 31^2 - 40 M\leq -315.
% 2 (34M - \sum_{i=1}^M v(\ell_i)) + \sum_{i=1}^M v(\ell_i) - 3M
\end{eqnarray}
Note that, if some lines have valency less than $25$, then the bounds \eqref{eqq2}, \eqref{eqq} improve accordingly
(as we exploit occasionally), and similarly for $M>128$.

If $\mF'$ does not contain any multiple components, then \eqref{eqq} directly leads to a contradiction as follows:
Write
\[
\mF' = \sum_{j=1}^r C_i
\]
for some distinct irreducible curves $C_i\subset X_5$.
Since $K_{X_5}=H$, adjunction gives 
\begin{eqnarray}
\label{eq:C^2}
C_i^2 \geq -2 - \deg(C_i).
\end{eqnarray}
Applied to the above decomposition of $\mF'$, this gives
\[
(\mF')^2 = \left(\sum_{j=1}^r C_j\right)^2 \geq \sum_{j=1}^r C_j^2 \geq -2r - \deg \mF' \geq -3\deg(\mF') \geq -81.
\]
This contradicts \eqref{eqq} by far, so $\mF'$ has to admit multiple components.
We shall now study the multiple components more precisely,
but before going into the details, we eliminate the lines from much of what is to follow.

\begin{lemm}
\label{lem:linear}
Any line contained in the support of $\mF'$ contributes positively to $(\mF')^2$.
\end{lemm}

\begin{proof}
Write $\mF'=\mF''+\mL$,
where $\mL$ is the sum of lines contained in the support of $\mF'$ (with multiplicities),
\[
\mL = \sum_{j=1}^s \ell^j,
\]
and $\mF''$ is an effective divisor on $X_5$ not containing any line in its support.
We compute, using \eqref{eq9},
\begin{eqnarray}
\label{eq:aux}
(\mF')^2  & = & \mF'.(\mF''+\mL) = (\mF'')^2 + \underbrace{\mF''.\mL}_{\geq 0} + \sum_{j=1}^s \underbrace{\ell^j.\mF'}_{\geq 9};
\end{eqnarray}
that is,
\begin{eqnarray}
\label{eq:L}
(\mF')^2
& \geq & (\mF'')^2 + 9\deg(\mL)
\end{eqnarray}
as claimed.
\end{proof}

%\begin{rem}
%One can also show without too much difficulty that any line has multiplicity at most three in $\mF'$,
%but we will not need this in the sequel.
%\end{rem}

Thanks  to Lemma \ref{lem:linear},
we can  reduce our considerations for $\mF'$ to satisfy \eqref{eqq}
to an investigation of the non-linear components of $\mF'$, i.e.~we shall often work with $\mF''$.
Sometimes we will even use linear components to our advantage, using the positive contribution
in \eqref{eq:L}. In fact, this can be sharpened substantially:

\begin{rem}
\label{rem:linear}
More precisely, any line $\ell$ contained in the support of $\mF'$ satisfies
\[
9 \leq \ell.\mF' = \underbrace{\ell.\mL}_{\leq  \deg(\mL)-4} + \;\; \ell.\mF'' \;\;  \Longrightarrow \;\; \ell.\mF'' \geq 13 - \deg(\mL).
\]
In particular, $\mL.\mF''\geq \deg(\mL)(13-\deg(\mL))$. Plugging into \eqref{eq:aux} causes \eqref{eq:L} to improve to
\begin{eqnarray}
\label{eq:LL}
(\mF')^2 \geq (\mF'')^2 + \max\{9\deg(\mL), \deg(\mL)(22-\deg(\mL))\}.
\end{eqnarray}
\end{rem}

The next result will be crucial in determining the top multiplicity of the components of $\mF'$.
For later reference, we first note it in the general set-up.

\begin{lemm}
\label{lem:multi}
Let $D$ denote an effective divisor with $D^2<0$.
Then the support of $D$ contains a component $C$
whose multiplicity $N$ satisfies
\begin{eqnarray}
\label{eq:multi}
N\cdot C^2 \leq\epsilon
\deg(C) \;\;\; \text{ where } \;\; \epsilon =  \frac{D^2}{\deg(D)}<0.
\end{eqnarray}
\end{lemm}

\begin{proof}
Write 
\[
D = \sum_i N_i C_i
\]
for distinct irreducible curves $C_i$.
Assume that \eqref{eq:multi} does not hold for either component of $D$.
Then
\[
D^2 \geq \sum_i N_i^2 C_i^2 > \sum_i N_i(\epsilon \deg(C_i)) = \epsilon\deg(D) = D^2,
\]
giving the required contradiction.
\end{proof}

We shall use this lemma repeatedly.
For our key case of $D=\mF'$, we note the following consequence.
%(which we will apply analogously to other suitable effective divisors).

\begin{lemm}
\label{lem:N}
The support of $\mF'$ contains a non-linear component $C$
whose multiplicity $N$ satisfies
\begin{eqnarray}
\label{eq:multi2}
N\cdot C^2 <-11\deg(C).
\end{eqnarray}
\end{lemm}

\begin{proof}
By Lemma \ref{lem:linear} and its proof, it suffices to consider the divisor $\mF''$ 
without linear components of degree at most $27$.
Then apply Lemma \ref{lem:multi} to $\mF''$ and simplify.
\end{proof}

\begin{rem}
In special situations, this argument can be strengthened using \eqref{eq:L}, \eqref{eq:LL}.
Also we could insert any constant $\epsilon<35/3$ instead of $11$ in \eqref{eq:multi2},
but presently this would not have any impact on our arguments.
\end{rem}

\begin{lemm}
\label{lem:conic}
The support of $\mF'$ contains an irreducible conic of multiplicity $N\geq 6$
or a twisted cubic curve of multiplicity $N\geq 7$.
\end{lemm}

\begin{proof}
Obviously the two given cases satisfy the inequality \eqref{eq:multi2} from Lemma \ref{lem:N}.
In comparison, the only remaining case in degree at most $3$,
a plane cubic curve $C\subset X_5$, has $C^2=-3$ by adjunction,
so \eqref{eq:multi2} would give multiplicity at least $12$, exceeding the degree of $\mF'$.

For $\deg(C)\geq 4$, the inequality \eqref{eq:C^2} eliminates 
all cases in combination with \eqref{eq:multi2}.
%except for $C$ being a twisted rational quartic curve $C$ of multiplicity $N=7$ in $\mF'$,
%i.e.~$\mF'=7C$ for degree reasons and $M=127$.
%But then \eqref{eq9} implies, for any line $\ell\subset X_5$, 
%\[
%\ell.C\geq 2 \;\; \Longrightarrow \;\; \ell.\mF'\geq 14 \;\; \Longrightarrow v(\ell) \;\; \leq 20.
%\]
%This allows us to improve \eqref{eqq2} by a summand $5M$, so that \eqref{eqq} 
%gives a contradiction.
\end{proof}

\begin{lemm}
\label{lem:9}
Any conic has multiplicity at most $9$ in $\mF'$.
\end{lemm}

\begin{proof}
For degree reasons, the multiplicity $N$ of a conic $Q$ in $\mF'$ is at most $13$.
Recall from Corollary \ref{cor88} that the valency of $Q$ satisfies
\[
v(Q) \leq 88,
\]
so there certainly is a line $\ell$ off $Q$.
But then $N\geq 10$ would imply
\[
\ell.\mF' = \ell.(\mF'-NQ) \leq \deg(\mF'-NQ) \leq 27-2N\leq 7,
\]
contradicting \eqref{eq9}.
\end{proof}

Before continuing, we record the following useful consequence:

\begin{lemm}
\label{lem:4l}
Any line has multiplicity at most $3$ in $\mF'$.
\end{lemm}

\begin{proof}
Let $\ell\subset X_5$ be of multiplicity $N$ in $\mF'$. 
First  assume that $N\geq 5$. Then use \eqref{eq9} and compute
\[
9\leq \ell.\mF' = -3N + \ell.(\mF'-N\ell) \leq -3N+\deg(\mF'-N\ell) \leq 27-4N.
\]
The claim follows immediately. It remains to discuss the case $N=4$.
Write $\mF'=4\ell+\mR$
with 
\[
\deg(\mR)\leq 23 \;\;\; \text{ and } \;\;\; \mR^2\leq -315+48-72 = -339
\]
by \eqref{eq:LL}. Then Lemma \ref{lem:multi} implies
that either $\mR$ contains a line of multiplicity at least $5$
(contradicting the above),
or a conic of multiplicity $m\geq 8$.
Hence we obtain
\[
\mF' = 4\ell + mQ + \mR', \;\;\; \deg(\mR')\leq 23-2m \leq 7.
\]
By Corollary \ref{cor88} and  Lemma \ref{lem:25},
there is a line $\ell'\subset X_5$ which intersects neither $\ell$ nor $Q$.
Hence $\ell'.\mF'\leq \deg(\mR') \leq 7$,
contradicting \eqref{eq9}.
This completes the proof of Lemma \ref{lem:4l}.
\end{proof}

%Before continuing, we record the following useful consequence:
%
%\begin{lemm}
%\label{lem:4l}
%Any line has multiplicity at most $4$ in $\mF'$.
%\end{lemm}
%
%\begin{proof}
%Let $\ell\subset X_5$ of multiplicity $N$ in $\mF'$. Use \eqref{eq9} and compute
%\[
%9\leq \ell.\mF' = -3N + \ell.(\mF'-N\ell) \leq -3N+\deg(\mF'-N\ell) \leq 27-4N.
%\]
%The claim follows immediately.
%\end{proof}

\begin{rem}
With a little more work,
one can improve Lemma \ref{lem:4l}
to show that any line has multiplicity at most $2$ in $\mF'$,
but we will not need this in what follows.
\end{rem}

\section{Proof of Theorem \ref{thm2}}
\label{s:pf}

We are now in the position to attack the proof of Theorem \ref{thm2}.
To this end, we have to rule out all possible configurations with an $N$-fold irreducible conic $Q$ supported on $\mF'$
($N=6,\hdots,9$), or with an $N$-fold twisted cubic $C$ ($N=7,8,9$).
We proceed by a case-by-case analysis of the residual effective divisor 
\[
\mR=\mF'-NQ, \deg(\mR)\leq 27-2N, \;\;
\text{ resp. } \;\; \mR=\mF'-NC, \deg(\mR)\leq 27-3N.
\]

\subsection{Multiple conic case}
\label{s:Q}

Following Lemma \ref{lem:conic} and Lemma \ref{lem:9},
we start by assuming that $\mF$ contains an irreducible conic $Q$ 
of multiplicity $N=6,\hdots,9$ (and no conic of higher multiplicity).
We proceed by a case-by-case analysis.

\subsubsection{$\boldsymbol{N=9}$}
\label{ss:9Q}

%Write $\mF'=9Q+\mR$ with $\deg(\mR)\leq 10$.
Recall from Remark \ref{rem88} that
\[
Q.\sum_{i=1}^M \ell_i \leq 88.
\]
Compared with $Q.\mF=62$, this yields 
\[
Q.\mR\geq 10
\;\; \text{ and } \;\;
(\mF')^2 \geq -324+180+\mR^2 = -144 +\mR^2.
\]
Hence \eqref{eqq} implies $\mR^2\leq -171$.
One might like to continue disregarding the linear components
as in Lemma \ref{lem:N}, but here the positive contribution from linear components
might already have been captured in $Q.\mR\geq 10$, so we have to content ourselves with applying Lemma \ref{lem:multi}
to $\mR$.
It follows that there is a seven-fold line $\ell$,
but this contradicts Lemma \ref{lem:4l}.
  %or a five-fold conic $Q'\neq Q$ 
%contained in the support of $\mR$.
%In the latter case, $\mF'=9Q=5Q'$,
%but then for \eqref{eq9} to hold, any line off $Q$ has to intersect $Q'$ with multiplicity two.
%Then Corollary \ref{cor88} and Lemma \ref{lem:25} give the required contradiction by way of
%% $M\leq 91$ and the required contradiction.
%%On the other hand, if there is a line $\ell$ with multiplicity at least 4 (or 2, in fact) in $\mR$,
%%then the same argument shows that all lines off $\ell$ have to meet $Q$,
%%i.e.
%\[
%M\leq v(\ell) + v(Q) \leq 113.
%\]

\subsubsection{$\boldsymbol{N=8}$}
\label{ss:8Q}

%We write $\mF'=8Q=\mR$ with $\deg(\mR)\leq 12$
%and 
We deduce as in \ref{ss:9Q} that $Q.\mR\geq 6$.
Hence $\mR^2\leq -155$,
and Lemma \ref{lem:multi} shows as before that 
$\mR$ contains a five-fold line 
%$\ell$ or a six-fold conic $Q'\neq Q$
%(since a five-fold conic can obviously not be complemented by an effective divisor  $\mR'$
%with $\deg(\mR')\leq 2$ and $(\mR')^2\leq -15$).
%Then the multiple linear component 
which is ruled out by Lemma \ref{lem:4l}.

%while $\mF'=8Q+6Q'$ would force
%\[
%\ell.\mF'\geq 12 \;\;\; \text{ for any line } \; \ell\subset X_5.
%\]
%As this implies $v(\ell)\leq 22$, the impact on \eqref{eqq2}
%easily rules out this case.

\subsubsection{$\boldsymbol{N=7}$}
\label{ss:7Q}

Consider the residual divisor $\mR$ with $\deg(\mR)\leq 13$.
As before, we infer $Q.\mR\geq 2$ and $\mR^2\leq -147$.
It follows from  Lemma \ref{lem:multi} that $\mR$ contains either a $4$-fold line $\ell$ 
(which is ruled out by Lemma \ref{lem:4l}) 
or a $6$-fold conic $Q'\neq Q$.
In the latter case, either $M=129$ and the bound \eqref{eqq} improves to
\[
(\mF')^2 \leq -355
\]
which is impossible to attain,
or $\mF'=7Q+6Q'+\ell$ for some line $\ell$ whose positive contribution
by Remark \ref{rem:linear} cannot be compensated for.

\subsubsection{$\boldsymbol{N=6}$}

The  residual divisor $\mR$ has degree $\deg(\mR)\leq 15$ and $\mR^2\leq -171$.
As in \ref{ss:7Q},
Lemma \ref{lem:multi} implies that $\mR$
contains 
a line of multiplicity $4$ 
(which again is excluded by Lemma \ref{lem:4l})
or another conic $Q'\neq Q$
of multiplicity $m\geq 6$.
In the latter case,
since $m=7$ has already been covered in \ref{ss:7Q},
it remains to realize that for $m=6$, 
Lemma \ref{lem:multi} applied to the residual divisor again leads to 
3-fold line and subsequently to a contradiction
by the positive contribution from \eqref{eq:L}.

\subsection{Multiple twisted cubic case}
\label{s:C}

Continuing the present line of arguments based on Lemma \ref{lem:conic},
we distinguish three cases depending on the multiplicity $N=7,8,9$ of the twisted cubic $C$ in $\mF'$.

\subsubsection{$\boldsymbol{N=7}$}
\label{ss:7C}

Here $\mR^2\leq -70$ and $\deg(\mR)\leq 6$, 
so Lemma \ref{lem:multi} 
implies that $\mR$ contains a 4-fold line.
Hence Lemma \ref{lem:4l} gives the required contradiction.
%so arguing for $\mR$ as in Lemma \ref{lem:N}
%directly gives a contradiction.

\subsubsection{$\boldsymbol{N=8}$}
\label{ss:8C}

In this case, $\deg(\mR)\leq 3$ with no obvious further restrictions.
We proceed by distinguishing how $\mR$ decomposes.

If $\mR$ contains a line,
then its positive contribution following Remark \ref{rem:linear} gives a contradiction.

If $\mR$ is an irreducible conic, then $M=129$, 
and the improved bound $\mF'^2 \leq -355$ from \eqref{eqq} cannot be reached.
Similarly, of course, if $\mR=0$.

Hence $\mR$ is an irreducible cubic (plane or twisted)
which thus meets no more than 98 lines (by looking at $\mR.\mF$).
All the remaining lines (at least 30 in number)
have to meet $C$ with multiplicity two for \eqref{eq9} to hold.
But then their valency drops by 7 to $v(\ell)=18$.
As this occurs for at least 30 lines, we get a big correction term of $-210$ to \eqref{eqq}
which is impossible to beat.

\subsubsection{$\boldsymbol{N=9}$}

We have $\mF'=9C$ and $M=128$.
Then \eqref{eq9} implies that any line $\ell\subset X_5$ meets $C$,
and for $C.\mF=93$ to hold,
$C$ has to intersect exactly 10 lines with multiplicity two, say $\ell_1\hdots,\ell_{10}$.
We turn to the effective divisor $D$ from \ref{ss:tcc}.
By Proposition \ref{prop:double}, there is a decomposition
\[
D = \sum_{i=1}^M \ell_i + \sum_{j=1}^{10} \ell_j + mC + D'.
\]
For degree reasons, this implies $m\leq 9$.
In comparison,
the intersection product reads
\[
99 = C.D = 158 - 5m + C.D'
\]
whence $m\geq 12$, contradiction.

\subsection{Proof of Theorem \ref{thm2}}

To conclude, let us wrap up the proof of Theorem \ref{thm2}.
Assuming that $M>127$, we deduced in Lemma \ref{lem:conic}
that $\mF$ contains a conic with multiplicity $N\geq 6$,
or a twisted cubic with multiplicity $N\geq 7$.
On the other hand, $N\leq 9$ by Lemma \ref{lem:9} resp.~for degree reasons.
Then the considerations in Sections \ref{s:Q}, \ref{s:C} successively ruled out all configurations
which might have fitted \eqref{eqq},
thus completing the proof of Theorem \ref{thm2}.
\qed

\begin{rem}
By inspection, the results of the preceding four sections are totally geometric,
i.e.~they essentially only require that \eqref{eq:fibration-deg-d} has a smooth fiber,
and that the flecnodal divisor does not degenerate.
%(both of which certainly hold true in positive characteristic if $p> d$).
%Only the last bound in Proposition \ref{prop} has to be replaced by $M\leq 52$,
%since the Picard number is only bounded by $b_2=53$.
\end{rem}


\begin{thebibliography}{99}

\bibitem{bpvh}  Barth, W., Hulek, K., Peters, C., van de Ven, A.:
\emph{Compact complex surfaces.} Second
edition, Erg. der Math. und ihrer Grenzgebiete, 3. Folge, Band 4. Springer, Berlin, 2004.



\bibitem{eisenbud-harris-3264}  Eisenbud, D., Harris, J.,  \emph{3264 and All That:
A Second Course  in Algebraic Geometry.}  Cambridge University Press, 2016.
%Draft available at
%  http://isites.harvard.edu/fs/docs/icb.topic720403.files/book.pdf, 2012.



\bibitem{sarti} Boissi\'ere, S., Sarti, A.: \emph{Counting lines on surfaces.} Ann. Sc. Norm. Super. Pisa, Cl. Sci. {\bf 6} (2007), 39--52.


\bibitem{clebsch} Clebsch,  A.: 
\emph{Zur Theorie der algebraischen Fl\"achen.}
 Journal reine angew. Math. {\bf 58} (1861), 93--108. 

\bibitem{D-2016} Degtyarev, A. I.: \emph{Lines in supersingular quartics.} Preprint (2016), arXiv:1604.05836.

\bibitem{dis-2015}  Degtyarev, A. I., Itenberg, I.,  Sert\"oz, A. S.: \emph{Lines on quartic surfaces.} 
Math. Ann. {\bf 368}  (2017), 753--809.

%\bibitem{GPS01}
%W.~Decker, G.-M.~Greuel, G.~Pfister, and H.~Sch\"onemann,
%\newblock {\sc Singular} {3-1-2} --- {A} computer algebra system for polynomial computations.
%\newblock {http://www.singular.uni-kl.de} (2010).



\bibitem{dimca} Dimca, A.: \emph{Singularities and topology of hypersurfaces}. Universitext. Springer-Verlag, New York, 1992. 

%
%\bibitem{dumnicki}  Dumnicki, M., Harbourne, B., Szemberg, T., Tutaj-Gasi\'nska, H.: \emph{Linear subspaces, symbolic powers and Nagata type conjectures.} Adv. Math. {\bf 252} , 471--491 (2014). 


\bibitem{kollar}   Koll\'ar, J.:
\emph{Szem\'eredi-Trotter-type theorems in dimension 3.} Adv. Math. {\bf 271} (2015), 30--61.

% \bibitem{maple}  Maple 2016
%
\bibitem{miyaoka-classical} Miyaoka, Y: \emph{The Maximal Number of Quotient Singularities on Surfaces with Given Numerical Invariants.} Math. Ann. {\bf 268}  (1984), 159--171. 

\bibitem{miyaoka} Miyaoka, Y:  
\emph{Counting Lines and Conics on a Surface.} Publ. RIMS, Kyoto Univ. {\bf 45} (2009), 919--923.

\bibitem{ramsdisjoint} Rams, S.: 
\emph{Projective surfaces with many skew lines.} Proc. Amer. Math. Soc. {\bf 133} (2005), 11--13.

\bibitem{ramsschuett}  Rams, S., Sch\"utt, M.: \emph{The Barth quintic surface has Picard number 41.} Ann. Sc. Norm. Super. Pisa Cl. Sci {\bf 13} (2014),
533--549.

\bibitem{rs-advgeom} Rams, S., Sch\"utt, M.: \emph{On quartics with lines of the second kind.}  Adv. Geom. {\bf 14} (2014), 735--756.

\bibitem{RS} Rams, S., Sch\"utt, M.: \emph{64 lines on smooth quartic surfaces.}  Math. Ann. {\bf 362} (2015), 679--698. 

\bibitem{rs-2014}  Rams, S., Sch\"utt, M.:
\emph{112 lines on smooth quartic surfaces (characteristic 3).} Quart. J. Math. {\bf 66} (2015), 941--951. 

\bibitem{rs-2}
Rams, S., Sch\"utt, M.:
\emph{At most 64 lines on smooth quartic surfaces (characteristic 2)}. Nagoya Math. J. (to appear).
  
% Preprint (2015), arXiv: 1512.01358.
%    SŁAWOMIR RAMS (a1) and MATTHIAS SCHÜTT (a2) (a3)
%        https://doi.org/10.1017/nmj.2017.21

\bibitem{salmon2}  Salmon, G.: \emph{A treatise on the analytic geometry of three dimensions. Vol. II.} Fifth edition. Longmans and Green, London, 1915.




\bibitem{Segre}
Segre, B.:
\emph{The maximum number of lines lying on a quartic surface}.
Quart. J. Math., Oxford
Ser. {\bf 14} (1943), 86--96.

\bibitem{veniani-phd}   Veniani, D. C.: \emph{The maximum number of lines lying on a K3 quartic surface}. 
Math. Z. {\bf 285} (2017), 1141-1166.

%\bibitem{voloch}
%Voloch, F.:
%\emph{Surfaces in $\PP^3$ over finite fields}.
%Topics in Algebraic and Noncommutative Geometry,
%Contemp.~Math.~{\bf 324} (2003), 219--226.





% \bibitem{h} J.~Harris, \emph{Algebraic geometry: A First
%     Course}. GTM 133, Springer (1995).


% \bibitem{hs-a} R.~Hartshorne, \emph{Ample subvarieties of algebraic varieties}, Springer Lecture
% Notes in Math. {\bf 156}, 1970.

% \bibitem{hs}  R.~Hartshorne, \emph{Algebraic geometry}. Graduate Texts
%   in Mathematics {\bf 52}, Springer-Verlag, New York, 1977. 

% \bibitem{rams-hab} S.~Rams,
% \emph{Defect and Hodge numbers of hypersurfaces.}
%  Adv. Geom. \textbf{8} (2008), 257-288.



% \bibitem{todd} T.~G.~Room, \emph{The geometry of determinatal loci}.  Cambridge University Press, Cambridge 1938. 
\end{thebibliography}
\end{document}